\documentclass[AMA,STIX1COL]{WileyNJD-v2}

\articletype{Article Type}%

\received{<day> <Month>, <year>}
\revised{<day> <Month>, <year>}
\accepted{<day> <Month>, <year>}

\usepackage{graphicx}

\usepackage{amsmath,amssymb,euscript,latexsym}

\usepackage{tikz}

\usepackage{cases}
\usepackage{epstopdf,graphicx}
\usepackage{hyperref}

\newcommand{\Sym}{\mathrm{Sym}}
\newcommand{\Skew}{\mathrm{Skew}}

\newcommand{\SO}{\operatorname{SO}(3)}
\newcommand{\MI}{{\mathbb I}}

\begin{document}

\title{On Controller Design for Systems on Manifolds in  Euclidean Space}

\author{Dong Eui Chang}

\authormark{Dong Eui Chang}

\address{\orgdiv{Electrical Engineering}, \orgname{Korea Advanced Institute of Science and Technology}, \orgaddress{\state{Daejeon}, \country{Korea}}}

\corres{Dong Eui Chang. \email{dechang@kaist.ac.kr}}

\presentaddress{Electrical Engineering, KAIST, 291 Deahak-ro, Yuseong-gu, Daejeon, 34141, Korea.}

\abstract[Summary]{A new method is developed to design controllers in Euclidean space for systems defined on manifolds. The idea is to embed the state-space manifold $M$ of a given control system  into some Euclidean space  $\mathbb R^n$,  extend the system from $M$ to the ambient space $\mathbb R^n$, and modify it outside $M$ to add transversal stability to $M$ in the final dynamics in $\mathbb R^n$. Controllers are designed for the final system in the ambient  space $\mathbb R^n$. Then,  their restriction to $M$ produces  controllers for the original system on  $M$.  This method has the  merit that only one single global Cartesian coordinate system in the ambient  space $\mathbb R^n$ is used for controller synthesis, and any controller design method in $\mathbb R^n$, such as the linearization method, can be globally applied for the controller synthesis.   The proposed method is successfully applied to the  tracking problem for the following two benchmark systems: the fully actuated rigid body system and the quadcopter drone system.}


\maketitle

\section{Introduction}
Many control systems are defined on manifolds that are not homeomorphic to Euclidean space, where we  use the term `Euclidean space' to mean some $\mathbb R^n$ space, not imposing any metric on it. The geometric, or coordinate-free, approach has been developed to deal with those systems without being dependent on the choice of coordinates.\cite{Bl07,BuLe04,Sa10} However, a state-space manifold often appears as an embedded manifold in  Euclidean space and the control system naturally extends from the manifold to the ambient Euclidean space: one example is the free rigid body system on $\SO \times \mathbb R^3$ which  naturally extends to  $\mathbb R^{3\times 3} \times \mathbb R^3$.  In such a case, it might be advantageous to use one single global Cartesian coordinate system in the ambient Euclidean space to design controllers for the original system on the manifold, eliminating the necessity to use  rather complex tools from differential geometry or multiple local coordinate systems. For example, in the case of the free rigid body system, neither adding nor subtracting two rotation matrices is allowed in the geometric approach  partly because the result does not lie on $\SO$, which may be mathematically orthodox, but would discourage control engineers from understanding or applying the geometric results.  Since any two rotation matrices, as $3\times 3$ matrices, can be  conveniently added or subtracted in $\mathbb R^{3\times 3}$, there is no reason to refrain from carrying out such  basic and convenient operations as additions and subtractions. Moreover, since one can utilize one single global Cartesian coordinate system in the ambient Euclidean space $\mathbb R^{3\times 3}$, he is free from  such discontinuities  as those that  often occur due to the switching of local coordinate systems and  chart-wise designed control laws.  As such, in this paper we propose a new method that is an  alternative to both the geometric approach, which adheres to differential geometric tools, and the classical approach, which employs local coordinates such as Euler angles for rigid bodies. 

A brief summary of the proposed method is provided as follows.  Given a control system $\Sigma_M$ whose dynamics evolve on a manifold $M$, we embed $M$ into some Euclidean space $\mathbb R^n$ and extend the system $\Sigma_M$ to a system $\Sigma_{\mathbb R^n}$ whose dynamics evolve in $\mathbb R^n$ or conservatively in a neighborhood of $M$ in $\mathbb R^n$. We then  legitimately modify the extended system $\Sigma_{\mathbb R^n}$ outside $M$ to add transversal stability  to $M$ while  the original dynamics on $M$ are kept intact. It follows that  $M$ becomes an attractive invariant manifold of the resulting system denoted $\tilde \Sigma_{\mathbb R^n}$.  We  apply any   controller design  method  available  in Euclidean space to design controllers  for $\tilde \Sigma_{\mathbb R^n}$ in $\mathbb R^n$  for stabilization of a point on $M$ or tracking of a reference trajectory on $M$, and then restrict the controllers to $M$ which yield controllers for the original system $\Sigma_M$ on $M$ for the stabilization or tracking on $M$.  To showcase this method,   the linearization  technique in $\mathbb R^n$  is chosen in this paper  to design tracking controllers although  we could alternatively apply other techniques  available in $\mathbb R^n$  such as  homogeneous approximation,\cite{Ha67} model predictive control,\cite{BoBeMo17}  iterative learning control,\cite{ScMuDa12} differential flatness,\cite{Le09} etc.

The  theory of embedding of manifolds in Euclidean space has a long history in mathematics, including  several famous  theorems  such as the Nash embedding theorems\cite{Na54,Na56} and the Whitney embedding theorem.\cite{Bo02} The embedding technique has been also   applied in control theory. For example, it was used  to produce a simple proof of  the Pontryagin maximum principle on manifolds,\cite{Ch11} and was  combined with the transversal stabilization technique to yield  feedback-based structure-preserving numerical integrators for simulation of  dynamical systems.\cite{ChJiPe16}
  A series of relevant works  have been made by  Maggiore and his collaborators on local transverse feedback linearizability of control-invariant submanifolds and virtual holonomic constraints.\cite{MaCo13,NiFuMa10,NiMa08}  The focus of Maggiore is placed on creation of a submanifold for a given system and its transversal stabilization via feedback for {\it path-following} controller synthesis,      whereas our work in this paper is focused on  embedding  and extending a state space manifold of a given system into  Euclidean space and its transversal stabilization for {\it tracking controller} synthesis. Moreover, our method has the merit to use one single global Euclidean coordinate system whereas the method by Maggiore does not. Another merit of our method is its openness to   accommodate   any existing control method developed in Euclidean space.

The paper is organized as follows. Section 2 is devoted to embedding into Euclidean space, transversal stabilization,   tracking controller design via linearization, and their application to the  rigid body system and the quadcopter drone system. Several tracking controllers are proposed for the two systems, and  the exponential convergence of their tracking error dynamics is rigorously proven  and numerical simulations are carried out to demonstrate the controllers' good tracking ability and  robustness to unknown disturbances.   The paper is concluded in Section 3.  The contributions of the paper are summarized as follows: 1.  the development of a new controller design methodology with the embedding and transversal stabilization technique which allows to convert difficult control problems on a manifold to  tractable control problem in Euclidean space and to use one single global Euclidean coordinate system in controller synthesis; and 2.  the design of exponentially tracking controllers with the developed method  for the rigid body system and the quadcopter system which are designed via linearization in ambient Euclidean space  but are still  expressed geometrically, i.e. in a coordinate-free manner. It is noted that a  presentation of    preliminary results  was given at the 56th IEEE Conference on Decision and Control.

\section{Main Results}
\subsection{Mathematical Preliminaries}
The usual Euclidean inner product is exclusively used for vectors and matrices in this paper, i.e.
\[
\langle A, B \rangle = \sum_{i,j}A_{ij}B_{ij} = \operatorname{tr}(A^TB)
\]
for any two matrices  of equal size.  The norm induced from this inner product, which is called the Frobenius or Euclidean norm, is exclusively used for vectors and matrices. Let $\Sym$ and $\Skew$ denote the symmetrization operator and the skew-symmetrization operator, respectively, on square matrices, which are defined by
\[
\Sym (A) = \frac{1}{2}(A+A^T), \quad \Skew (A) = \frac{1}{2}(A-A^T)
\]
for any square matrix $A$.
Then,
\[
A  = \Sym (A) + \Skew(A), \quad \langle \Sym (A), \Skew (A) \rangle = 0.
\]
 Namely, 
\[
\mathbb R^{n\times n} = \Sym (\mathbb R^{n\times n} ) \oplus \Skew (\mathbb R^{n\times n} )
\]
with respect to the Euclidean inner product. Let $[\, , ]$ denote the usual matrix commutator that is defined by $[A,B] = AB-BA$ for any pair of square matrices $A$ and $B$ of equal size.
It is easy to show that
\begin{align*}
[\Sym (\mathbb R^{n\times n} ), \Skew (\mathbb R^{n\times n} )] &\subset \Sym (\mathbb R^{n\times n} ),\\
[\Skew (\mathbb R^{n\times n} ), \Skew (\mathbb R^{n\times n} )] &\subset \Skew (\mathbb R^{n\times n} ),\\
[\Sym (\mathbb R^{n\times n} ), \Sym (\mathbb R^{n\times n} )] &\subset \Skew (\mathbb R^{n\times n} ).
\end{align*}
In other words,  $[A,C] = [A,C]^T$ for all  $A = A^T \in \mathbb R^{n\times n}$ and $C = -C^T \in \mathbb R^{n\times n}$;  $[B,C] = -[B,C]^T$   for all  $B = -B^T \in \mathbb R^{n\times n}$ and  $C = -C^T \in \mathbb R^{n\times n}$; and $[B,C] = -[B,C]^T$   for all  $B = B^T \in \mathbb R^{n\times n}$ and  $C = C^T \in \mathbb R^{n\times n}$.
Let $\SO$ denote the set of all $3\times 3$ rotation matrices, which is defined as $\SO = \{ R\in \mathbb R^{3\times 3} \mid R^T R - I= 0, \det R>0\}$.  Let $\mathfrak{so}(3)$ denote the set of all $3\times 3$ skew symmetric matrices, which is defined as $\mathfrak{so}(3) = \{ A \in \mathbb R^{3\times 3} \mid A^T+ A = 0 \}$. The hat map $\wedge : \mathbb R^3 \rightarrow \mathfrak{so}(3)$ is defined by
\[
\hat \Omega = \begin{bmatrix}
0 & -\Omega_3 &\Omega_2 \\
\Omega_3 & 0 & -\Omega_1\\
-\Omega_2 & \Omega_1 & 0
\end{bmatrix}
\]
for $\Omega = (\Omega_1, \Omega_2,\Omega_3) \in \mathbb R^3$.  The inverse map of the hat map is called the vee map and denoted  $\vee$ such that $(\hat \Omega)^\vee = \Omega$ for all $\Omega \in \mathbb R^3$ and $(A^\vee)^\wedge = A$ for all $A\in \mathfrak{so}(3)$.  
\begin{lemma}\label{lemma:prelim}
1. $\langle R A, R B\rangle = \langle AR, BR\rangle  = \langle A, B\rangle$   for all $R \in \SO$ and $A,B \in \mathbb R^{3\times 3}$.

2. $\max_{R_1, R_2 \in \SO}\| R_1 - R_2\| = 2\sqrt 2$. 

3. $\langle \hat u, \hat v \rangle = 2 \langle u, v\rangle$ for all $u, v \in \mathbb R^3$.

4. $[\hat u, \hat v]= (u \times v)^\wedge$  and $\hat u v = u \times v$ for all $u, v \in \mathbb R^3$.
\end{lemma}
 Given a function $f: A \rightarrow B$ and a subset $C$ of $B$, the set $f^{-1}(C)$ is defined as $f^{-1}(C)  = \{ a\in A \mid f(a) \in C\}$. In particular, when $C$ consists of a single point, say $c$, we just write $f^{-1}(c)$ to mean $f^{-1}(\{c\})$. Every function and manifold is assumed to be smooth in this paper unless stated otherwise. Stability, stabilization and tracking are all understood to be local unless globality is stated explicitly.  The reader is referred to the book by Bloch\cite{Bl07} for more information on manifolds. 

\subsection{Embedding in Euclidean Space and Transversal Stabilization}

\subsubsection{Theory}
Let $M$ be an $m$-dimensional regular manifold in $\mathbb R^n$, where $m < n$. 
Consider a control system $\Sigma_M$ on $M$ given by
\begin{equation}\label{Sigma:M}
\Sigma_M: \quad \dot x = X(x,u), \quad x \in  M, u \in \mathbb R^k.
\end{equation}
Notice that 
\begin{equation}\label{X:in:TM}
X(x,u) \in T_xM  \quad \forall x \in  M, u \in \mathbb R^k,
\end{equation}
where $T_xM$ denotes the tangent space to $M$ at $x$.  Suppose that there is a control system  $\Sigma_{\mathbb R^n}$ on $\mathbb R^n$ given by
\begin{equation}\label{nonlinear:ambient}
\Sigma_{\mathbb R^n}: \quad \dot x = X_e(x,u), \quad x \in \mathbb R^n, u \in \mathbb R^k,
\end{equation}
that satisfies
\begin{equation}\label{Xe:X}
X_e(x,u) = X(x,u) \quad \forall x \in M, u \in \mathbb R^k.
\end{equation}
In other words, $\Sigma_{\mathbb R^n}$ is an extension of $\Sigma_M$ to $\mathbb R^n$ and $\Sigma_M$ becomes a restriction of $\Sigma_{\mathbb R^n}$ to $M$.  By \eqref{X:in:TM} and \eqref{Xe:X}, $M$ is an invariant manifold of $\Sigma_{\mathbb R^n}$.

Suppose that there is a function $\tilde V: \mathbb R^n \rightarrow \mathbb R_{\geq 0}$ such that 
\begin{equation}\label{M:zero:level:V}
M = \tilde V^{-1}(0)
\end{equation}
and 
\begin{equation}\label{nt:V}
\nabla \tilde V (x) \cdot X_e(x,u) = 0
\end{equation}
for all $x \in \mathbb R^n$ and $u \in \mathbb R^k$.
With this function,  construct a  system $\tilde \Sigma_{\mathbb R^n}$ in $\mathbb R^n$ as
\begin{align}\label{Sigma:tilde}
\tilde \Sigma_{\mathbb R^n}: \quad \dot x& = \tilde X_e(x,u), \quad x \in \mathbb R^n, u \in \mathbb R^k,
\end{align}
where the vector field $\tilde X_e$ is defined by
\begin{equation}\label{def:X:tilde}
\tilde X_e(x,u) =X_e(x,u) -\nabla \tilde V(x)  \quad \forall x \in \mathbb R^n, u \in \mathbb R^k.
\end{equation}
Since every point in $M$ is a minimum point of $V$,  $\nabla V(x)$ vanishes on $M$ identically. Hence, by \eqref{Xe:X} and \eqref{def:X:tilde}
\begin{equation}\label{tilde:Xe:X}
\tilde X_e(x,u) = X(x,u) \quad \forall x \in M, u \in \mathbb R^k.
\end{equation}
In other words, the system $\tilde \Sigma_{\mathbb R^n}$ coincides with the original system $\Sigma_M$ on $M$. Hence, $M$ is an invariant manifold of $\tilde \Sigma_{\mathbb R^n}$ as well. Along any flow of $\tilde \Sigma_{\mathbb R^n}$
\begin{equation}\label{dVdt:less}
\frac{d}{dt} \tilde V= \nabla \tilde V \cdot  (X_e - \nabla \tilde V) = - \|\nabla \tilde V\|^2 \leq 0
\end{equation}
by \eqref{nt:V}. 

\begin{theorem}\label{theorem:attractive:M}
If there are positive numbers $b$  and $r$ such that 
\begin{equation}\label{hypothesis:lambda}
b  \tilde V (x) \leq  \| \nabla \tilde V(x) \|^2 
\end{equation}
for all $x \in \tilde V^{-1}([0,r))\subset \mathbb R^n$, then $\tilde V^{-1}([0,r))$ is positively invariant for $\tilde \Sigma_{\mathbb R^n}$ and every flow of $\tilde \Sigma_{\mathbb R^n}$ starting  in $\tilde V^{-1}([0,r))$ converges to $M$ as $t\rightarrow \infty$. In particular, $\tilde V(x(t))  \leq \tilde V(x(0)) e^{-b t} $ for all $t\geq 0$ and $x(0) \in \tilde V^{-1}([0,r))$.
\begin{proof}
It follows from   \eqref{dVdt:less} and \eqref{hypothesis:lambda} that for any initial state $x(0) \in  \tilde V^{-1}([0,r))$, $\tilde V(x(t)) \leq \tilde V(x(0)) e^{-b t} < r e^{-b t}$ for all $t\geq 0$, where $x(t)$ is the flow of  $\tilde \Sigma_{\mathbb R^n}$ starting from $x(0)$. It implies that  $\tilde V^{-1}([0,r))$ is a positively invariant set of $\tilde \Sigma_{\mathbb R^n}$ and that  $\lim_{t\rightarrow\infty} \tilde  V(x(t)) = 0$. From \eqref{M:zero:level:V} and the continuity of  $\tilde V$, it follows that $x(t)$ converges to $M$ as $t\rightarrow \infty$.
\end{proof}
\end{theorem}

 The following corollary shows a typical situation in which  to construct such a function $\tilde V$ that satisfies \eqref{M:zero:level:V}, \eqref{nt:V} and the hypothesis of Theorem \ref{theorem:attractive:M}. 
 \begin{corollary}\label{corollary:construction:V}
 Suppose that there is a function $F: \mathbb R^n \rightarrow \mathbb R^{n-m}$  such that $M = F^{-1}(0)$; that there is an open set $S \subset \mathbb R^n$ such that $M \subset S$ and every point in $S$ is a regular point of $F$;  that  $DF(x) \cdot X_e(x,u) = 0$ for all $(x,u) \in S\times \mathbb R^k$; and that there is a number $c>0$ such that the smallest singular value of $\| DF(x)\|$ is larger than $c$ for every $x\in S$.  Suppose also that $\tilde V(x) = F(x)^T K F(x)$ is used to define the system $\tilde \Sigma_{\mathbb R^n}$ in \eqref{Sigma:tilde} and   \eqref{def:X:tilde}, where $K$ is an $(n-m)\times (n-m)$ positive definite symmetric matrix.
 Then, there is an open set $W$ in $\mathbb R^n$ with $M \subset W$ such that every trajectory of $\tilde \Sigma_{\mathbb R^n}$  starting in $W$ remains in  $W$ for all future time and exponentially converges to $M$ as $t \rightarrow \infty$.
 \begin{proof}
 Let $\tilde V(x) = F(x)^T K F(x)$, where $K$ is an $(n-m)\times (n-m)$ positive definite symmetric matrix.  Then, $\nabla \tilde V(x) = 2DF(x)^T KF(x)$ in column vector form. It is easy to show that this function $\tilde V$ satisfies \eqref{M:zero:level:V} and \eqref{nt:V} for all $(x,u) \in S \times \mathbb R^k$.  By hypothesis, $ \|\nabla  \tilde V(x)\|  = \|  2DF(x)^T KF(x)\|\geq  2c \|K F(x)\| \geq 2c \lambda_{\rm min}(K) \|F(x)\|$ for all $x\in S$. Hence, for any $x\in S$, $\|\tilde V(x)\| \leq \lambda_{\rm \max}(K) \|F(x)\|^2 \leq  (\lambda_{\rm \max}(K) / 4c^2 \lambda_{\rm min}(K)^2) \|\nabla \tilde V(x)\|^2$. Let $b = 4c^2 \lambda_{\rm min}(K)^2/ \lambda_{\rm \max}(K)$ and choose a number $r>0$ such that $\tilde V^{-1}([0,r)) \subset S$ which is possible due to continuity of the function $\tilde V$. With these numbers $b$ and $r$,  the hypothesis of Theorem \ref{theorem:attractive:M} holds true. Hence, by Theorem \ref{theorem:attractive:M}, $W:= \tilde V^{-1}([0,r))$ is a positively invariant region of attraction for  $\tilde \Sigma_{\mathbb R^n}$, and  $\tilde V(x(t)) \leq \tilde V(x(0))e^{-bt}$ for all  $x(0) \in W$ and $t\geq 0$. This inequality implies that  
 \[
 \|F(x(t))\| \leq A \|F(x(0))\| e^{- bt/2}
 \]
 for all $x(0) \in W$ and all $t\geq 0$, where $A = \sqrt{\lambda_{\rm max}(K)/\lambda_{\rm min}(K)}$. Since every point of $W$ is a regular point of $F$, $F(x)$ can be used as part of local coordinates such that $M = \{F(x) = 0\}$. Hence,  the above inequality shows that the convergence  of $x(t)$ to $M$ is exponential.
 \end{proof}
 \end{corollary}

Our  goal is to design controllers for the system $\Sigma_M$ whose dynamics evolve on the manifold $M$. Since the  system $\tilde \Sigma_{\mathbb R^n}$ in $\mathbb R^n$
 coincides with $\Sigma_M$ on $M$, and $M$ is an invariant manifold of $\tilde\Sigma_{\mathbb R^n}$, we can first design controllers for $\tilde \Sigma_{\mathbb R^n}$ in one single global Cartesian coordinate system for $\mathbb R^n$ and then restrict them to $M$ to come up with controllers for the original system $\Sigma_M$. This method becomes much more tractable when $M$ is an attractive invariant manifold of $\tilde \Sigma_{\mathbb R^n}$, which is guaranteed by the hypothesis in Theorem \ref{theorem:attractive:M}. Notice that the size of the region of attraction of $M$ for the $\tilde \Sigma_{\mathbb R^n}$ dynamics is immaterial since the set $\mathbb R^n \backslash M$ is not a region of interest but only an auxiliary ambient region in which we take full advantage of the Euclidean structure of  $\mathbb R^n$. 

\subsubsection{Application to the Rigid Body System}
As a main example throughout the paper, we use the  free rigid body system with full actuation whose equations of motion are  given by
\begin{subequations}\label{rigid:original}
\begin{align}
\dot R &= R\hat \Omega, \\
\dot \Omega &= \MI^{-1} ( \MI \Omega \times \Omega) + \MI^{-1} \tau,
\end{align}
\end{subequations}
where $(R,\Omega) \in \SO \times  \mathbb R^3 \subset \mathbb R^{3\times 3} \times \mathbb R^3$ is the state vector consisting of a rotation matrix $R$ and a body angular velocity vector $\Omega$; $\tau \in \mathbb R^3$ is the control torque; and $\mathbb I$ is the moment of inertial matrix of the rigid body. From here on, we regard the system \eqref{rigid:original} as a system defined on  $\mathbb R^{3\times 3} \times \mathbb R^3$, treating $R$ as a $3\times 3$ matrix.  It is then easy to verify that $\SO \times \mathbb R^3$ is an invariant set of \eqref{rigid:original}, i.e. every flow starting in $M$ remains in $M$ for all $t\in \mathbb R$.
Assume that the full state of the system is available, which allows us to apply the following controller 
\begin{equation}\label{calcelling:controller}
\tau = \MI (u - \MI^{-1}(\MI \Omega \times \Omega))
\end{equation}
to transform the above system to
\begin{subequations}\label{rigid:eq}
\begin{align}
\dot R &= R\hat \Omega, \label{R:eq}\\
\dot \Omega &= u, \label{Omega:eq}
\end{align}
\end{subequations}
where $u$ is the new control vector. Note that $\SO \times \mathbb R^3$ is an invariant set of \eqref{rigid:eq}.  Let $
\operatorname{GL}^+(3) = \{ R \in \mathbb R^{3\times 3} \mid \det R >0\}$ 
 and define a function $\tilde V: \operatorname{GL}^+(3) \times \mathbb R^3\subset \mathbb R^{3\times 3} \times\mathbb R^3 \rightarrow \mathbb R_{\geq 0}$ by 
\begin{equation}\label{def:V:tilde}
\tilde V(R,\Omega) = \frac{k_e}{4}\|R^TR - I\|^2,
\end{equation}
where $k_e>0$ is a constant. It is easy to verify that $\tilde V^{-1}(0) = \SO\times \mathbb R^3$ and 
\begin{equation}\label{nabla:V:tilde}
\nabla_R \tilde V = -k_eR(R^TR - I), \quad \nabla_\Omega \tilde V = 0.
\end{equation}
With this function $\tilde V$,  the modified rigid body system corresponding to \eqref{Sigma:tilde} and \eqref{def:X:tilde} is computed as
\begin{subequations}\label{rigid:tilde:eq}
\begin{align}
\dot R &= R\hat \Omega - k_eR(R^TR - I), \label{R:s:eq}\\
\dot \Omega &= u, \label{Omega:s:eq}
\end{align}
\end{subequations}
where $(R,\Omega) \in\operatorname{GL}^+(3) \times \mathbb R^3 \subset \mathbb R^{3\times 3} \times \mathbb R^3$.

We now show that Theorem \ref{theorem:attractive:M} holds in the rigid body case.
\begin{lemma}\label{lemma:attractive:M:rigid} 
There are numbers $b>0$ and $r>0$ such that 
\[
b \tilde V(R,\Omega) \leq \| \nabla \tilde V(R,\Omega)\|^2
\]
for all $(R,\Omega) \in \tilde V^{-1}([0,r))$. 
\begin{proof}
Define an auxiliary function $f: \operatorname{GL}^+(3) \rightarrow \mathbb R_{\geq 0}$ by
\[
f(R) = \frac{k_e}{4} \|R^TR -I\|^2      
\]
for $R \in \operatorname{GL}^+(3)$.
Take any sufficiently small $\epsilon>0$ such that every $A \in \mathbb R^{3\times 3}$ satisfying $
\|A - I\| \leq \epsilon$
is invertible. Let $r = k_e\epsilon^2/4$.
 Then, if $R\in f^{-1}([0,r])$,  $\|R^TR -I\| \leq \epsilon$,
so $R^TR$ is invertible, which implies that $R$ is also invertible. Hence, $f^{-1}([0,r]) \subset \operatorname{GL}^+(3)$.
For each $i = 1,2,3$ and any $R\in \mathbb R^{3\times 3}$,
\[
\sum_{j = 1}^3 R_{ji}^2 =  | (R^TR)_{ii} |  \leq \|R^TR\| 
\]
which implies
\begin{equation}\label{R:bound}
\| R\|^2 =  \sum_{i=1}^3\sum_{j = 1}^3 R_{ji}^2 \leq 3 \|R^TR\|
\end{equation}
for any $R\in \mathbb R^{3\times 3}$. 
Hence for any $R \in f^{-1}([0,r])$,
\[
\|R^TR\| \leq \|R^TR - I \|  + \|I\| \leq \epsilon + 3,
\]
which implies by \eqref{R:bound} that $\|R\| \leq \sqrt{3\epsilon + 9}$
for all $R \in f^{-1}([0,r])$. It follows that $f^{-1}([0,r])$ is compact in $\mathbb R^{3\times 3}$, being closed and bounded. Since  the matrix inversion operation is continuous, the image of $f^{-1}([0,r])$ under matrix inversion is also compact. Hence, there is a number $M>0$ such that $\|R^{-1}\| \leq M$
for all $R \in  f^{-1}([0,r])$. Hence, for any $(R,\Omega) \in \tilde V^{-1}([0,r))$
\begin{align*}
\|R^TR -I\| = \|  R^{-1}R(R^TR-I)\|\leq \|R^{-1}\| \|R(R^TR-I)\| \leq M  \|R(R^TR-I)\|
\end{align*}
which implies $b \tilde V(R, \Omega) \leq \|\nabla \tilde V \|$ for all $(R,\Omega ) \in \tilde V^{-1}([0,r))$ by \eqref{def:V:tilde} and \eqref{nabla:V:tilde},
where  $
b = 4k_e/M^2$.
This completes the proof.
\end{proof}
\end{lemma}

\begin{theorem}\label{theorem:attractive:M:rigid:exponential} 
There is a number $r>0$ such that every trajectory of \eqref{rigid:tilde:eq} starting in  $\tilde V^{-1}([0,r))$ remains in $\tilde V^{-1}([0,r))$ for all future time and converges exponentially to $\SO \times \mathbb R^3$ as $t\rightarrow \infty$.  
\begin{proof}
Pick such numbers $b$ and $r$ as in the statement of Lemma \ref{lemma:attractive:M:rigid}.
By Lemma \ref{lemma:attractive:M:rigid} and Theorem \ref{theorem:attractive:M},  every trajectory of \eqref{rigid:tilde:eq} starting in  $\tilde V^{-1}([0,r))$ remains in $\tilde V^{-1}([0,r))$ for all future time and converges to $\SO \times \mathbb R^3$ as $t\rightarrow \infty$.  
Let $(R(t), \Omega (t))$ be an arbitrary trajectory staring in $\tilde V^{-1}([0,r))$ at $t=0$. Then, by Theorem \ref{theorem:attractive:M}, it satisfies
\[
\| R^T(t) R(t) - I \| \leq \|R^T(0) R(0) - I \| e^{-bt/2}
\]
for all $t\geq 0$. It follows that the convergence of $(R(t), \Omega (t))$ to $\SO \times \mathbb R^3$ is exponential since the $3\times 3$ zero matrix is a regular value of  the map $g: \operatorname{GL}^+(3)  \rightarrow \Sym(\mathbb R^{3\times 3})$ defined by $g(R) = R^T R - I$ such that   $\SO = \{ R\in  \operatorname{GL}^+(3)   \mid g(R) = 0\}$; refer to pp.22--23 of Guillemin and Pollack\cite{GuPo74} to see why the zero matrix is a regular value of $g$. 
\end{proof}
\end{theorem}

\begin{remark}
The technique of  embedding into  ambient Euclidean space and transversal stabilization was successfully tested in creating feedback integrators for structure-preserving numerical integration\cite{ChJiPe16}  of the dynamics of uncontrolled dynamical systems.  This technique is extended  to control systems in this paper. In particular,  Theorem \ref{theorem:attractive:M}, Corollary \ref{corollary:construction:V},  Lemma \ref{lemma:attractive:M:rigid} and Theorem  \ref{theorem:attractive:M:rigid:exponential}  in this paper are new and powerful so as  to guarantee exponential stability of $M$ in the transversal direction.
\end{remark}

\subsection{Tracking via Linearization in Ambient Euclidean Space}

\subsubsection{Theory}
Consider again the system $\tilde \Sigma_{\mathbb R^n}$ given in \eqref{Sigma:tilde} and its restriction $\Sigma_M$ to $M$ given in \eqref{Sigma:M}. Choose a reference trajectory $x_0: [0,\infty) \rightarrow M$ for $\Sigma_M$ on $M$ driven by a control signal $u_0: [0,\infty) \rightarrow \mathbb R^k$ so that 
\[
\dot x_0(t) = \tilde X(x_0(t), u_0(t)) \quad \forall t\geq 0.
\]
We can then linearize the ambient system $\tilde \Sigma_{\mathbb R^n}$ along the trajectory $(x_0(t), u_0(t))$  in $\mathbb R^n$ as follows:
\begin{equation}\label{Sigma:tilde:ell:t}
\tilde \Sigma^\ell_{\mathbb R^n} : \quad  \Delta \dot x = A(t) \Delta x+  B(t)\Delta u,
\end{equation}
where 
\[
A(t) = \frac{\partial \tilde X}{\partial x}(x_0(t),u_0(t)), \quad B(t) = \frac{\partial \tilde X}{\partial u}(x_0(t),u_0(t)) 
\]
and
\[
\Delta x = x - x_0 (t) \in \mathbb R^n, \quad \Delta u = u - u_0 (t) \in \mathbb R^k.
\]
Refer to Section 4.6 of Khalil\cite{Kh02} about the linearization technique. Notice that the above linearization does not require any use of local charts on the state-space manifold $M$. In that sense the above linearization is conducted {\it globally} along the reference trajectory in one global coordinate system in $\mathbb R^n$.  Also, in comparison with such a geometric linearization method as variational linearization in Lee et al.\cite{LeLeMc11} our Jacobian linearization is straightforward and simple to carry out.  The following lemma is trivial but useful:
\begin{lemma}\label{lemma:exp:tracking}
If $u = u(t,x)$ is an exponentially tracking controller for the ambient system $\tilde \Sigma_{\mathbb R^n}$ for the reference trajectory $x_0(t)$, then it is also an exponentially tracking controller for the system $ \Sigma_M$ on $M$ for the same reference trajectory.
\end{lemma}
The following theorem is an adaptation of Theorem 4.13 from the textbook by  Khalil\cite{Kh02}  in combination with Lemma \ref{lemma:exp:tracking} above. 
\begin{theorem}\label{theorem:Khalil}
Suppose that a linear feedback controller   $\Delta u = - K(t)\Delta x$ exponentially stabilizes the origin for the linearized system $\tilde \Sigma_{\mathbb R^n}^\ell$ in $\mathbb R^n$.  Let $B_r = \{ z \in \mathbb R^n \mid \| z\| <r\}$ for some $r>0$ and $f:  [0,\infty)  \times B_r \rightarrow \mathbb R$ be a function defined by
\[
f(t,z) = \tilde X( x_0(t) + z, u_0(t) - K(t) z) - \tilde X(x_0(t),u_0(t)).
\]
If the derivative $\frac{\partial f}{\partial z}(t,z)$ is bounded and Lipschitz on $B_r$ uniformly in $t$, then the controller
\[
u(t,x) = u_0(t)  - K(t) (x - x_0(t))
\]
enables the system $\Sigma_M$ on $M$  to track the reference trajectory $x_0(t)$ exponentially. 

\end{theorem}

Notice that the key point in the above theorem is that the controller for the system $\Sigma_M$ on $M$ is designed in the ambient Euclidean space $\mathbb R^n$.

\subsubsection{Application to the Rigid Body System}
We here apply Theorem \ref{theorem:Khalil} to the free rigid body system \eqref{rigid:tilde:eq}. Take a reference trajectory $(R_0(t), \Omega_0 (t)) \in \SO \times \mathbb R^3$ and the corresponding control signal $u_0(t)$ such that 
\begin{equation}\label{ref:traj:rigid}
\dot R_0(t) = R_0(t)\hat \Omega_0 (t), \quad \dot \Omega_0 (t) = u_0(t), \,\, \forall t \geq 0,
\end{equation}
which can be also understood as equations that define $\Omega_0(t)$ and $u_0(t)$ in terms of $R_0(t)$ and its time derivatives. Assume that  $(R_0(t), \Omega_0(t))$ and $u_0(t)$ are bounded  over the time interval $[0,\infty )$.
\begin{theorem}
The linearization of \eqref{rigid:tilde:eq} along the reference trajectory $(R_0(t),\Omega_0(t)) \in \SO \times \mathbb R^3$ and the reference control signal $u_0(t)$ is given by
\begin{subequations}\label{rigid:tilde:ell}
\begin{align}
\Delta \dot R &= \Delta R \hat\Omega_0+ R_0 \widehat{\Delta \Omega} - 2k_e R_0 \Sym (R_0^T\Delta R), \label{rigid:tilde:ell:a}\\
\Delta\dot \Omega &= \Delta u, \label{rigid:tilde:ell:b}
\end{align}
\end{subequations}
where
\[
\Delta R = R - R_0(t) \in \mathbb R^{3\times 3}, \quad \Delta \Omega = \Omega - \Omega_0(t) \in \mathbb R^3, \quad  \Delta u = u - u_0(t) \in \mathbb R^3.
\]
\begin{proof}
Equation \eqref{rigid:tilde:ell:a} can be easily derived by using the definition of derivative as follows. Let $c(s) = R_0 + s (R-R_0) = R_0 + s\Delta R$ and $d(s) = \Omega_0 + s (\Omega - \Omega_0) = \Omega_0 + s\Delta \Omega$, where $s\in \mathbb R$. Then
\begin{align*}
\left . \frac{d}{ds} \right |_{s=0}  (c(s)\widehat {d(s)} - kc(s)(c(s)^Tc(s) - I)) &=\Delta R \hat\Omega_0+ R_0 \widehat{\Delta \Omega}  - k_eR_0 (\Delta R^T R_0 + R_0^T\Delta R)\\
 &= \Delta R \hat\Omega_0+ R_0 \widehat{\Delta \Omega}  - 2k_eR_0\Sym(R_0^T\Delta R),
\end{align*}
which is equal to the expression on the right side of  \eqref{rigid:tilde:ell:a}.
\end{proof}

\end{theorem}

We now introduce a new matrix variable $Z$ replacing $\Delta R$ as follows:
\begin{equation}\label{ZR0R}
Z = R_0^T(t) \Delta R.
\end{equation}
 Let
\begin{equation}\label{def:Zs}
Z_s = \Sym (Z), \quad Z_k=\Skew (Z)
\end{equation}
such that 
\begin{equation}\label{def:Zk}
Z = Z_s + Z_k.
\end{equation}

\begin{lemma}
The system \eqref{rigid:tilde:ell} is transformed to
\begin{subequations}\label{Zs:Zk:Om:tv}
\begin{align}
\dot Z_s &= [Z_s, \hat \Omega_0] - 2k_e Z_s,\label{Zs:Zk:Om:tv:a}\\
\dot Z_k^\vee &= Z_k^\vee \times \Omega_0 + \Delta \Omega,  \label{Zs:Zk:Om:tv:b}\\
\Delta \dot \Omega &= \Delta u \label{Zs:Zk:Om:tv:c}
\end{align}
\end{subequations}
via the state transformation given in \eqref{ZR0R} -- \eqref{def:Zk}.
\end{lemma}
\begin{proof}
Differentiate  \eqref{ZR0R} with respect to $t$ and use   \eqref{ref:traj:rigid} -- \eqref{def:Zk} to obtain
\begin{align*}
\dot Z &= \dot R_0^T \Delta R + R_0^T\Delta \dot R \\
&= -\hat\Omega_0R_0^T\Delta R + R_0^T \Delta R \hat \Omega_0 + \widehat {\Delta \Omega} - 2k_e\Sym(R_0^T\Delta R)\\
&= [Z, \hat \Omega_0] + \widehat {\Delta \Omega} - 2k_e\Sym(Z)\\
&= [Z_s,  \hat \Omega_0]  + [Z_k,  \hat \Omega_0] + \widehat {\Delta \Omega} - 2k_e Z_s.
\end{align*}
Taking the symmetric and skew-symmetric parts, we get
\begin{align*}
\dot Z_s =  [Z_s,  \hat \Omega_0]  - 2k_e Z_s,\quad \dot Z_k =   [Z_k,  \hat \Omega_0] + \widehat {\Delta \Omega},
\end{align*}
where the second equation can be also written as \eqref{Zs:Zk:Om:tv:b} by Lemma \ref{lemma:prelim}.
This completes the proof. 
\end{proof}

\begin{proposition}\label{proposition:delta:1}
For any two matrices $K_P, K_D \in \mathbb R^{3\times 3}$ such that the matrix 
\begin{equation}\label{point:stab:Hurwitz:2}
\begin{bmatrix}
0 & I \\
-K_P & -K_D
\end{bmatrix}
\end{equation}
 is Hurwitz,  the controller
\begin{align}
\Delta u &= -K_P \cdot Z_k^\vee - K_D (Z_k^\vee \times \Omega_0 + \Delta \Omega )  - (Z_k^\vee \times \Omega_0 + \Delta \Omega )\times \Omega_0 - Z_k^\vee \times u_0 \label{Delta:u:rigid}
\end{align}
exponentially stabilizes the origin for the system \eqref{Zs:Zk:Om:tv}. 
\begin{proof}
Let us first show the exponential stability of the subsystem \eqref{Zs:Zk:Om:tv:a} that is decoupled from the rest of the system. Let $V(Z_s) = \|Z_s\|^2/2$. Along the trajectory of  \eqref{Zs:Zk:Om:tv}, $\frac{d}{dt}V = \langle Z_s, [Z_s, \hat\Omega_0]  - 2k_eZ_s\rangle = -2k_e \|Z_s\|^2 = -4k_eV$, where it is easy to show $\langle Z_s, [Z_s, \hat\Omega_0]  \rangle = 0$. Hence, $V(t) \leq e^{-4k_et}V(0)$ for all $t\geq 0$, or 
\begin{equation}\label{Zs:exp}
\|Z_s(t)\| \leq e^{-2k_et} \|Z_s(0)\|
\end{equation} 
for all $t\geq 0$ and $Z_s(0) \in \Sym(\mathbb R^{3\times 3})$, which proves exponential stability of $Z_s = 0$ for  \eqref{Zs:Zk:Om:tv:a}.

Differentiating \eqref{Zs:Zk:Om:tv:b} and substituting \eqref{Zs:Zk:Om:tv:c}  transforms the subsystem \eqref{Zs:Zk:Om:tv:b} and \eqref{Zs:Zk:Om:tv:c} to the following second-order system:
\[
\ddot Z_k^\vee = \dot Z_k^\vee \times \Omega_0 +Z _k^\vee \times u_0 + \Delta u
\]
since $\dot \Omega (t) = u_0(t)$. This second-order system
 is exponentially stabilized by the controller
\begin{equation}\label{Delta:u:another:form}
\Delta u = -K_P \cdot Z_k^\vee - K_D \dot Z_k^\vee -  \dot Z_k^\vee \times \Omega_0 - Z_k^\vee \times u_0,
\end{equation}
where the matrices $K_P, K_D \in \mathbb R^{3\times 3}$ are any matrices such that the matrix in \eqref{point:stab:Hurwitz:2} becomes Hurwitz. So, there are positive constants $C_1$ and $C_2$ such that 
\[
\|Z_k^\vee (t)\| + \|\dot Z_k^\vee (t)\| \leq C_1e^{-C_2t} (\|Z_k^\vee (0)\| + \|\dot Z_k^\vee (0)\| )
\]
for all $t\geq 0$ and $(Z_k^\vee (0), \dot Z_k^\vee (0)) \in \mathbb R^{3} \times  \mathbb R^{3}$.
Since $\Omega_0(t)$ is bounded by assumption, there is a constant $M>0$ such that $ \|\Omega_0(t)\| \leq M$ for all $t\geq0$. By  \eqref{Zs:Zk:Om:tv:b} and the triangle inequality,
\[
\|\dot Z_k^\vee (t) \| \leq   M\|Z_k^\vee (t)\| + \| \Delta \Omega (t)\|
\]
and
\[
 \|\Delta \Omega (t)\| \leq \| \dot Z_k^\vee (t)\| + M \|Z_k^\vee (t)\| 
\]
for all $t\geq 0$. It is then easy to show that
\begin{align}
\|Z_k^\vee (t)\| &+ \|\Delta \Omega (t)\|\leq C_3e^{-C_2t}(\|Z_k^\vee (0)\| + \|\ \Delta \Omega (0)\|)  \label{Zk:Om:exp}
\end{align}
for all $t\geq 0$ and $(Z_k^\vee (0), \Delta \Omega (0))\in \mathbb R^{3} \times  \mathbb R^{3}$, where $C_3 = C_1(1+M)^2$. Notice that the controller given in \eqref{Delta:u:another:form} is the same as the one in \eqref{Delta:u:rigid}. It  follows  
from \eqref{Zs:exp} and \eqref{Zk:Om:exp} that the controller  \eqref{Delta:u:rigid} exponentially stabilizes the origin for the system \eqref{Zs:Zk:Om:tv}. 
\end{proof}
\end{proposition}

\begin{remark}
The exponential stability of the subsystem \eqref{Zs:Zk:Om:tv:a} is a consequence of adding the term $- k_eR(R^TR - I)$ in \eqref{R:s:eq}, and it is  consistent with Theorem \ref{theorem:attractive:M:rigid:exponential}.
\end{remark}

The following proposition produces time-varying PID-like tracking controllers. 
\begin{proposition}\label{proposition:delta:2}
For any three matrices $K_P, K_D, K_I \in \mathbb R^{3\times 3}$ such that the polynomial 
\begin{equation}\label{PID:poly:2}
\det (\lambda^3 I + \lambda^2K_D + \lambda K_P + K_I) = 0
\end{equation}
is Hurwitz,  the controller
\begin{align}
\Delta u &= -K_P \cdot Z_k^\vee - K_D (Z_k^\vee \times \Omega_0 + \Delta \Omega )  - K_I\int_0^t   Z_k^\vee (\tau) d\tau  - (Z_k^\vee \times \Omega_0 + \Delta \Omega )\times \Omega_0 - Z_k^\vee \times u_0 \label{Delta:u:rigid:PID}
\end{align}
exponentially stabilizes the origin for the  system \eqref{Zs:Zk:Om:tv}. 
 \begin{proof}
Apply the controller  \eqref{Delta:u:rigid:PID} to the system \eqref{Zs:Zk:Om:tv}  and differentiate \eqref{Zs:Zk:Om:tv:b} three times to transform the closed-loop system \eqref{Zs:Zk:Om:tv} to 
\begin{align*}
\dot Z_s &= -2k_e Z_s,\\
\dddot Z_k^\vee &+ K_D \ddot Z_k^\vee + K_P\dot Z_k^\vee + K_I  Z_k^\vee = 0.
\end{align*}
It is easy to prove that this linear system is exponentially stable by the Hurwitz condition on the polynomial in \eqref{PID:poly:2}.
This proves  the proposition.
\end{proof}

\end{proposition}

The controllers proposed in \eqref{Delta:u:rigid} and \eqref{Delta:u:rigid:PID}  depend on the reference control signal $u_0(t)$. The following proposition  proposes one that is independent of $u_0(t)$.
\begin{proposition}\label{proposition:delta:3}
For any positive number $k_P$ and any positive definite symmetric matrix $K_D \in \mathbb R^{3\times 3}$, the  controller
\begin{equation}\label{delta:u:kpd}
\Delta u = -k_PZ_k^\vee - K_D \Delta \Omega
\end{equation}
exponentially stabilizes the origin for the  system \eqref{Zs:Zk:Om:tv}. 
\begin{proof}
Since the exponential stability of the subsystem \eqref{Zs:Zk:Om:tv:a} has been shown  in the proof of Proposition \ref{proposition:delta:1}, it remains to prove the exponential stability of the  subsystem \eqref{Zs:Zk:Om:tv:b} and \eqref{Zs:Zk:Om:tv:c} with the control law given above. 
Since $\Omega_0(t)$ is bounded by assumption, there is a number $M$ such that $\|\Omega_0(t)\| \leq M$ for all $t\geq 0$. Choose a number $\epsilon$ such that
\begin{equation}\label{eps:inequality}
0 < \epsilon < \min\left \{ \sqrt k_P,  \frac{4k_P\lambda_{\min} (K_D)}{4k_P + (M + \lambda_{\max} (K_D))^2} \right \}.
\end{equation}
Define two functions $V_1$ and $V_2$ by
\begin{align}
V_1 &= \frac{k_P}{2}\|Z_k^\vee\|^2 + \frac{1}{2}\|\Delta \Omega \|^2 + \epsilon \|Z_k^\vee\| \|\Omega\|,\label{def:V1:track} \\
V_2 &= \epsilon k_p\|Z_k^\vee\|^2 + (\lambda_{\min}(K_D) -\epsilon)\|\Delta \Omega\|^2- \epsilon (M+\lambda_{\max}(K_D))\|Z_k^\vee\|\|\Omega\|. \nonumber
\end{align}
These two functions are all positive definite quadratic functions of $(\|Z_k^\vee\|, \|\Omega\|)$ by \eqref{eps:inequality}, so there exists a constant $C>0$ such that
\begin{equation}\label{C2:V2:V3}
C V_1  \leq V_2.
\end{equation}
Define a function $V$ by
\begin{equation}\label{def:V:track}
V = \frac{k_P}{2} \|Z_k^\vee\|^2 + \frac{1}{2}\|\Delta \Omega\|^2 + \epsilon \langle Z_k^\vee,\Delta  \Omega  \rangle,
\end{equation}
which is a positive definite quadratic function of $(Z_k^\vee, \Delta \Omega)$ and satisfies
\begin{equation}\label{V1:V:V2}
 V \leq V_1.
\end{equation}
Along any trajectory of the subsystem \eqref{Zs:Zk:Om:tv:b} and \eqref{Zs:Zk:Om:tv:c} with the control \eqref{delta:u:kpd},
\begin{align*}
\frac{d}{dt}V &= k_P\langle Z_k^\vee, Z_k^\vee \times \Omega_0 + \Delta \Omega \rangle +\langle \Delta \Omega, u\rangle +\epsilon ( \langle Z_k^\vee \times \Omega_0 + \Delta \Omega, \Delta \Omega \rangle + \langle Z_k^\vee, u\rangle )\\
&\leq -\epsilon k_P \|Z_k^\vee \|^2 - (\lambda_{\min}(K_D) - \epsilon)\|\Delta \Omega\|^2  + \epsilon (M + \lambda_{\max}(K_D))\|Z_k^\vee\| \|\Delta\Omega\| \\
&=-V_2  \leq -CV_1 \leq -CV
\end{align*}
by \eqref{C2:V2:V3} and \eqref{V1:V:V2}. Hence, $V(t) \leq e^{-Ct}V(0)$ for all $t\geq 0$, which implies  that the closed-loop subsystem  \eqref{Zs:Zk:Om:tv:b} and \eqref{Zs:Zk:Om:tv:c}  is exponentially stable with the control \eqref{delta:u:kpd}. This completes the proof. 
\end{proof}
\end{proposition}

The following proposition is a variant of Proposition \ref{proposition:delta:3}.
\begin{proposition}\label{proposition:delta:4}
For any two positive numbers $k_P$ and $\epsilon$ and any positive definite symmetric matrix $K_D \in \mathbb R^{3\times 3}$ such that
\begin{equation}\label{inequal:eps:2}
0 < \epsilon < \min \left \{ \sqrt{k_P}, \frac{4k_P\lambda_{\min}(K_D)}{4k_P + (\lambda_{\max}(K_D))^2}\right \},
\end{equation}
 the controller
\begin{equation}\label{delta:u:kpd:Omega0}
\Delta u = -k_P Z_k^\vee - K_D \Delta \Omega - \epsilon (Z_k^\vee \times \Omega_0)
\end{equation}
exponentially stabilizes the origin for the system \eqref{Zs:Zk:Om:tv}. 
\begin{proof}
The exponential stability of \eqref{Zs:Zk:Om:tv:a}  has already been shown in the proof of Theorem \ref{proposition:delta:1}, so we now focus on the stability of \eqref{Zs:Zk:Om:tv:b}  and \eqref{Zs:Zk:Om:tv:c} with the feedback \eqref{delta:u:kpd:Omega0}. Consider the same function $V_1$ as that defined in \eqref{def:V1:track}. Let 
\begin{align*}
V_2 &= \epsilon k_p\|Z_k^\vee\|^2 + (\lambda_{\rm min}(K_D) - \epsilon) \|\Delta \Omega \|^2  -\epsilon \lambda_{\rm max}(K_D)\|Z_k^\vee\| \|\Delta \Omega\|.
\end{align*}
By \eqref{inequal:eps:2}, the two functions $V_1$ and  $V_2$ are both positive definite quadratic functions of $(\|Z_k^\vee\|, \|\Omega\|)$, so there exists a constant $C>0$  such that  \eqref{C2:V2:V3} holds. Consider the function $V$ defined in \eqref{def:V:track}, which  is a positive definite quadratic function of $(Z_k^\vee, \Delta \Omega)$ and  satisfies \eqref{V1:V:V2}. It is then straightforward to show that along any trajectory of the subsystem \eqref{Zs:Zk:Om:tv:b} and \eqref{Zs:Zk:Om:tv:c} with the control  \eqref{delta:u:kpd:Omega0},  $
\frac{d}{dt}V \leq -V_2  \leq -CV_1 \leq -C V$
by \eqref{C2:V2:V3} and \eqref{V1:V:V2}.
Hence, $V(t) \leq e^{-Ct}V(0)$ for all $t\geq 0$, which implies  that the closed-loop subsystem  \eqref{Zs:Zk:Om:tv:b} and \eqref{Zs:Zk:Om:tv:c}  is exponentially stable with the control \eqref{delta:u:kpd:Omega0}. This completes the proof.

\end{proof}
\end{proposition}

The following proposition essentially derives the control law in equation (13) of Lee et al. \cite{LeChEu17} which was derived using geometric control theory therein, but is easily derived  here with the linearized dynamics \eqref{Zs:Zk:Om:tv}. \begin{proposition}\label{proposition:delta:5}
For any $k_R>0$ and $k_\Omega>0$, the controller
\begin{equation}\label{LeChEu17:again}
\Delta u = -k_R Z_k^\vee - k_\Omega \Delta \Omega + \Delta \Omega \times \Omega_0
\end{equation}
exponentially stabilizes the origin for the system \eqref{Zs:Zk:Om:tv}.
\begin{proof}
Choose any number $\epsilon$ that satisfies $ 0 < \epsilon < \min \{ \sqrt{k_R}, {4k_Rk_\Omega}/{ (4 k_R + k_\Omega^2)}\}$.
Then, the function $V(Z_k^\vee, \Delta \Omega ) = k_R \|Z_k^\vee\|^2/2 + \epsilon \langle Z_k^\vee, \Delta \Omega \rangle +  \|\Delta \Omega \|^2/2 $ is a positive definite quadratic function of $(Z_k^\vee, \Delta \Omega)$. Along any flow of  \eqref{Zs:Zk:Om:tv:b} and \eqref{Zs:Zk:Om:tv:c}, the derivative of $V$ can be easily computed as $dV/dt = -\epsilon k_R \|Z_k^\vee\|^2 - \epsilon k_\Omega \langle Z_k^\vee, \Delta \Omega \rangle  - (k_\Omega - \epsilon) \|\Delta \Omega \|^2$, which can be easily shown to be a negative definite quadratic function of $(Z_k^\vee, \Delta \Omega)$, which proves the closed-loop exponential stability of the origin for the system \eqref{Zs:Zk:Om:tv}. 
\end{proof}
\end{proposition}

The following theorem puts together the five preceding propositions to provide various exponentially tracking controllers for the rigid body system \eqref{rigid:eq}.
\begin{theorem}\label{theorem:killing}
The following controller
\begin{equation}\label{THE:tracking:control:rigid}
u = u_0 + \Delta u,
\end{equation}
where $\Delta u$ is any of \eqref{Delta:u:rigid},  \eqref{Delta:u:rigid:PID}, \eqref{delta:u:kpd}, \eqref{delta:u:kpd:Omega0} and \eqref{LeChEu17:again} with 
\begin{equation}\label{Zk:killing}
Z_k = \Skew (R_0^T\Delta R)^\vee = \Skew(R_0^TR)^\vee,
\end{equation}
enables the  rigid body system   \eqref{rigid:eq} to track the reference trajectory $(R_0(t), \Omega_0(t))$ exponentially.
 \begin{proof}
By \eqref{ZR0R}, $\|\Delta R(t)\| = \|R_0(t)Z(t)\| = \|Z(t)\|$, so  exponential stability of \eqref{Zs:Zk:Om:tv} implies that of \eqref{rigid:tilde:ell}. Hence, this theorem follows from Theorem \ref{theorem:Khalil} and Propositions \ref{proposition:delta:1} --  \ref{proposition:delta:5}.
\end{proof}
\end{theorem}

\begin{remark}
As can be seen in \eqref{Zk:killing},  $Z_k$ can be computed without  computing $\Delta R = R- R_0(t)$. As a result, all the control laws for the rigid body system   \eqref{rigid:eq} on $\SO \times \mathbb R^3$ provided in Theorem \ref{theorem:killing} can be computed using matrix multiplications on $\SO$ although they have been derived with  $\Delta R$ in $\mathbb R^{3\times 3}$. In other words, all the control laws in Theorem \ref{theorem:killing} are intrinsic on $\SO \times \mathbb R^3$ though they are derived in the ambient Euclidean space $\mathbb R^{3\times 3} \times \mathbb R^3$. 
\end{remark}
\begin{remark}
One can observe that the subsystem \eqref{Zs:Zk:Om:tv:b} coincides with the $\dot \eta$ equation in (16)   in the paper by Lee  et al,\cite{LeLeMc11} where equation (16)   therein  is derived through so-called variational linearization. Since we have extended the rigid body system to ambient Euclidean space, our  linearization is the usual Jacobian linearization taken in Euclidean space, which is not only simpler than the variational one, but also allows us to  rigorously and easily apply the Lyapunov linearization method in one signle global Cartesian coordinate system with  the transversal dynamics   \eqref{Zs:Zk:Om:tv:a} taken into account. Also, thanks to the added term $-\nabla \tilde V$, the $Z_s$-subsystem   \eqref{Zs:Zk:Om:tv:a}, which is decoupled from the subsystem \eqref{Zs:Zk:Om:tv:b} and \eqref{Zs:Zk:Om:tv:c}, is  exponentially stable by itself. Without it, i.e. if $k_e =0 $, the $Z_s$-dynamics would be only neutrally stable, not enabling us to directly apply the Lyapunov linearization method.
\end{remark}

We carry out a simulation to show a good tracking performance of the controller  \eqref{THE:tracking:control:rigid} with \eqref{delta:u:kpd:Omega0}  for the rigid body system \eqref{rigid:eq} or  \eqref{rigid:tilde:eq} with $k_e=1$.
The control parameters are chosen as 
\[
k_P=4, \quad K_D = 2I, \quad \epsilon = 1.
\]
The reference trajectory $(R_0(t), \Omega_0(t)) \in \SO \times \mathbb R^3$ with the reference control signal $u_0(t) \in \mathbb R^3$ are chosen as 
\begin{align}
&R_0(t)=
\begin{bmatrix}
 \cos^2 t & ( 1+ \sin t)\cos t \sin t &  (\sin t -\cos^2 t)\sin t\\
 -\sin t \cos t & \cos^2t - \sin^3t & (1 + \sin t)\cos t \sin t \\
 \sin t & -\cos t \sin t & \cos^2 t
\end{bmatrix}, \label{tracking:rigid:R0}\\
&\Omega_0(t) = \begin{bmatrix}
-1-\sin t, & (-1 + \sin t)\cos t, &
-\sin t -\cos^2 t
\end{bmatrix}^T, \label{tracking:rigid:Omega0}\\
& u_0(t) = \dot \Omega_0(t)=
 \begin{bmatrix}
-\cos t, & \sin t + \cos^2 t -\sin^2 t,& -\cos t + 2\cos t\sin t
\end{bmatrix}^T,\label{tracking:rigid:u0}
\end{align}
which satisfy \eqref{ref:traj:rigid}. Notice that if the reference trajectory $R_0(t)$ is parameterized by the $Z-Y-X$ Euler angles, then the parameterization will become singular at $t = \pi/2 + k\pi$, $k \in \mathbb Z$. Hence, the use of Euler angles for  controller design  is not desirable. The initial condition is chosen as
\[
R(0) = \exp (0.99\pi \hat e_2), \quad \Omega (0) =  (-1,-1,-1),
\]
where $R(0)$ is a rotation around $e_2 = (0,1,0)$ through $0.99\pi$ radians. The initial orientation tracking error is almost $2\sqrt 2$ that is the maximum possible orientation error. The tracking errors  are plotted in Figure \ref{figure.tracking:rigid}, which shows a good tracking performance of the controller for the {\it nonlinear} system \eqref{rigid:eq}.

We now carry out a simulation to compare the controller \eqref{THE:tracking:control:rigid} and \eqref{delta:u:kpd:Omega0}  with the controller proposed by Lee\cite{Le11} which is modified for the system  \eqref{rigid:eq}  as follows:
\[
u_{\rm Lee} = -k_R e_R - k_\Omega e_\Omega - \hat \Omega R^TR_0\Omega_0 + R^TR_0\dot \Omega_0,
\]
where
\[
e_R = \frac{1}{\sqrt{1+ \operatorname{trace}(R_0^TR)}}\Skew(R_0^TR)^\vee, \quad e_\Omega = \Omega - R^TR_0\Omega_0.
\]
For the controller  \eqref{THE:tracking:control:rigid} with \eqref{delta:u:kpd:Omega0}, we use the parameter values: $k_P=4$, $K_D = 2I$ and $\epsilon = 1$. To make a fair comparison, we choose for the controller $u_{\rm Lee}$ the following parameter values: $k_R = 4$ and $k_\Omega =2$. The two controllers are applied to the system \eqref{rigid:eq} with the initial condition $R(0) = \exp (0.9\pi \hat e_2)$ and $\Omega (0) =  (-1,-1,-1)$ for the reference trajectory given in \eqref{tracking:rigid:R0} -- \eqref{tracking:rigid:u0}. The simulation results are plotted in Figure \ref{figure.tracking:rigid:Lee}. We can see that there is a difference between the two controllers in the  transient response. 
The controller by Lee  initially performs better than our controller in attitude tracking but it has a large overshoot in angular velocity tracking and has a huge initial value of control, which is due to the nonlinear term $ 1/\sqrt{1+ \operatorname{trace}(R_0^TR)}$ present in Lee's controller, $u_{\rm Lee}$.  After about $t=5$, both controllers behave similarly, and the  responses of the system are similar to each other.  From these observations, we can draw the conclusion that our linear controller \eqref{THE:tracking:control:rigid} with \eqref{delta:u:kpd:Omega0} is on par with the nonlinear controller $u_{\rm Lee}$ by Lee. However, our controller has been easily obtained with a linear technique whereas the controller by Lee was obtained with a nonlinear technique that is not as easy to use as the linear technique.

\begin{figure}[t]
\begin{center}
\includegraphics[scale = 0.25]{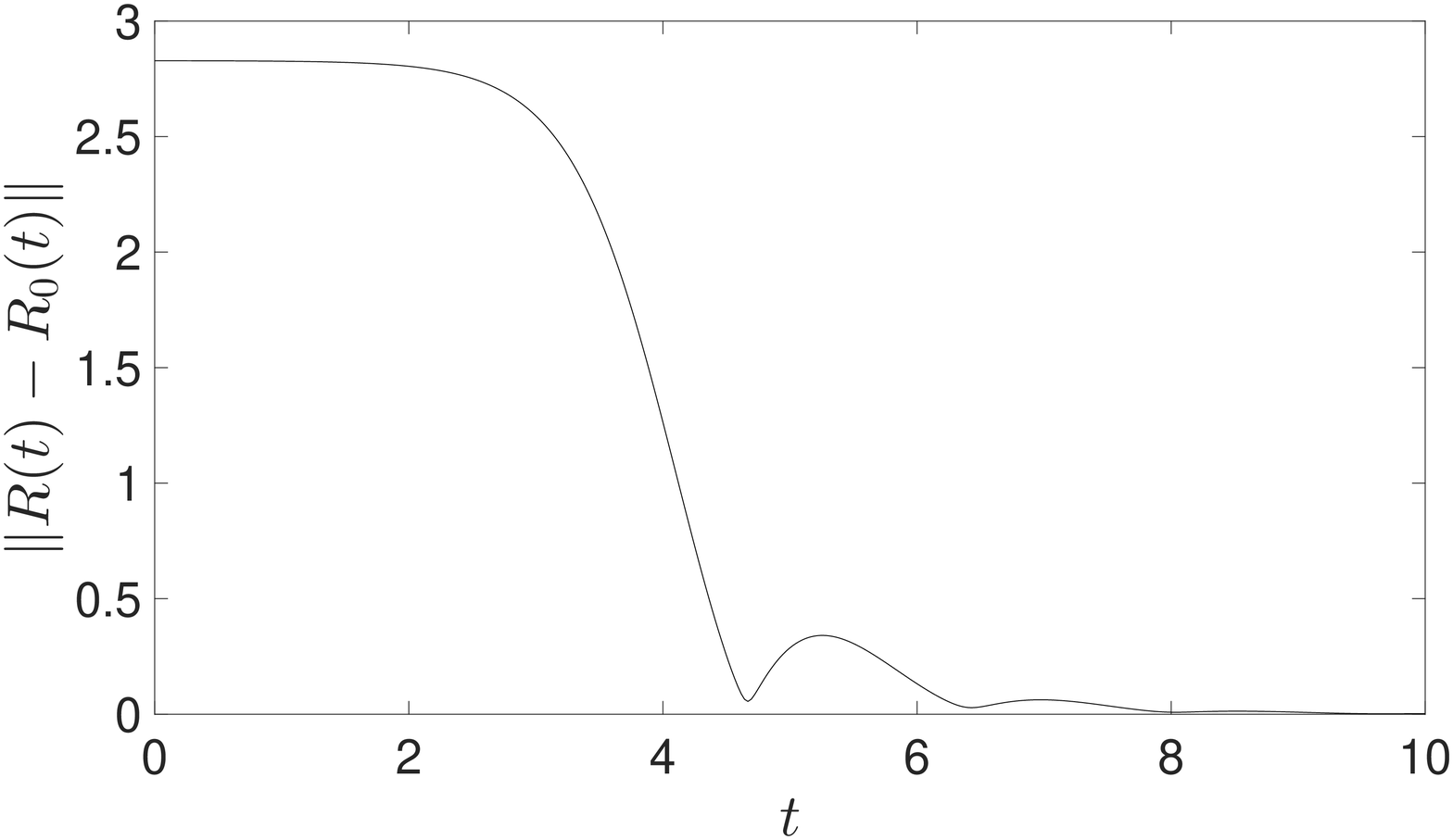}~~~\includegraphics[scale = 0.25]{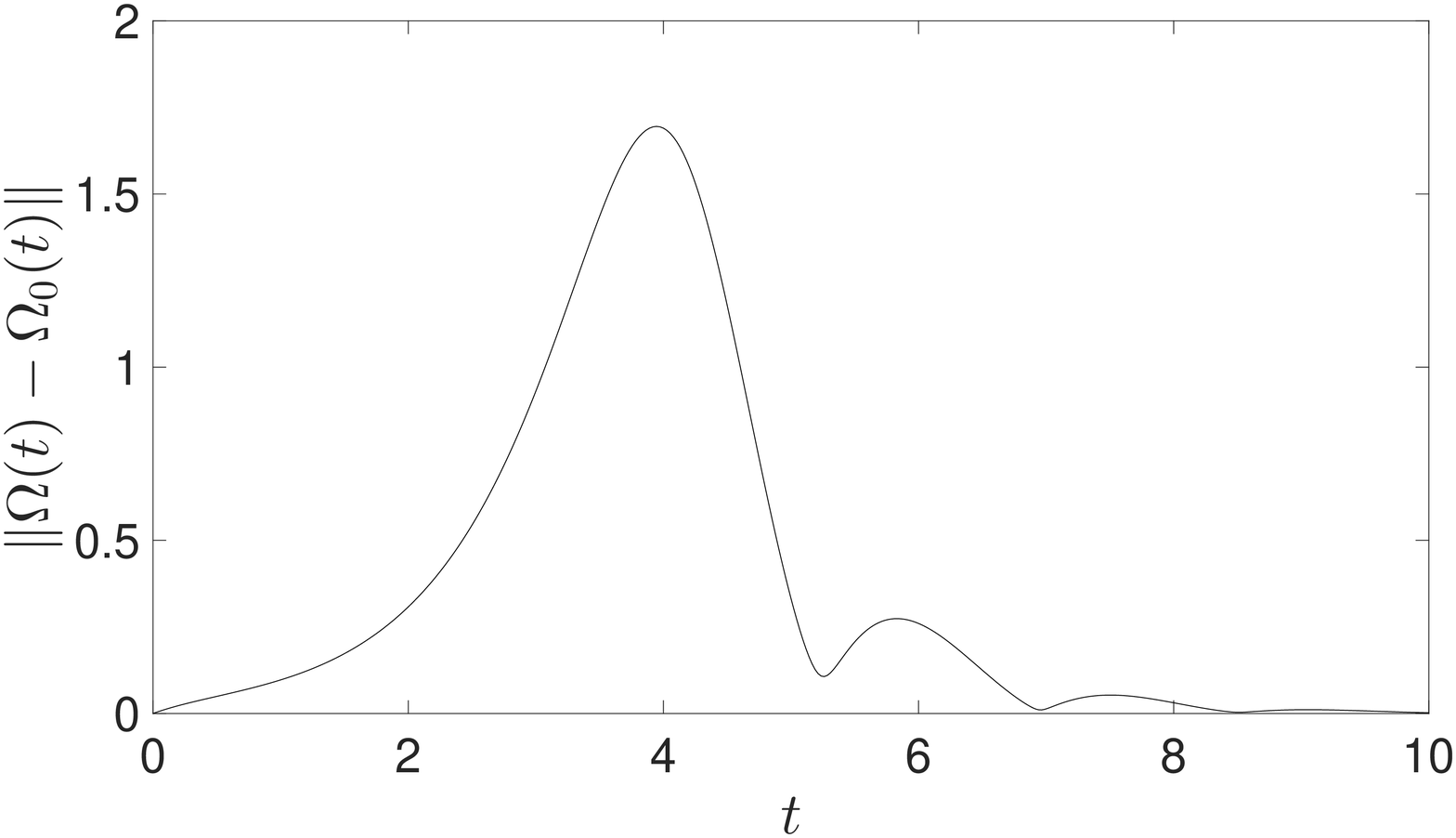}
\end{center}
\caption{\label{figure.tracking:rigid}  The simulation result for  tracking of the reference $(R_0(t), \Omega_0(t))$ by the linear controller  \eqref{THE:tracking:control:rigid} with \eqref{delta:u:kpd:Omega0}  for the rigid body system.}
\end{figure}

\begin{figure}[t]
\begin{center}
\includegraphics[scale = 0.25]{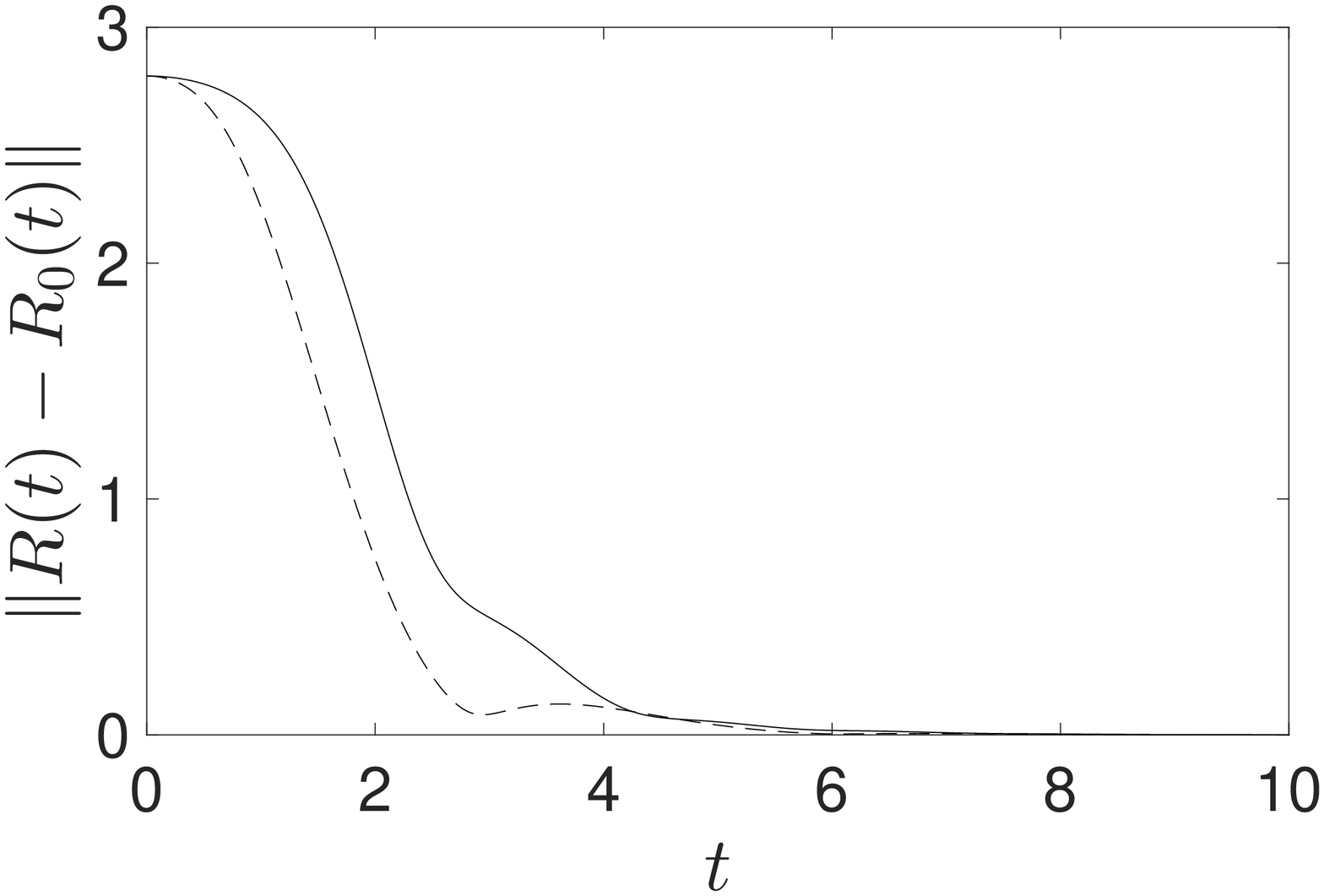}~~~\includegraphics[scale = 0.25]{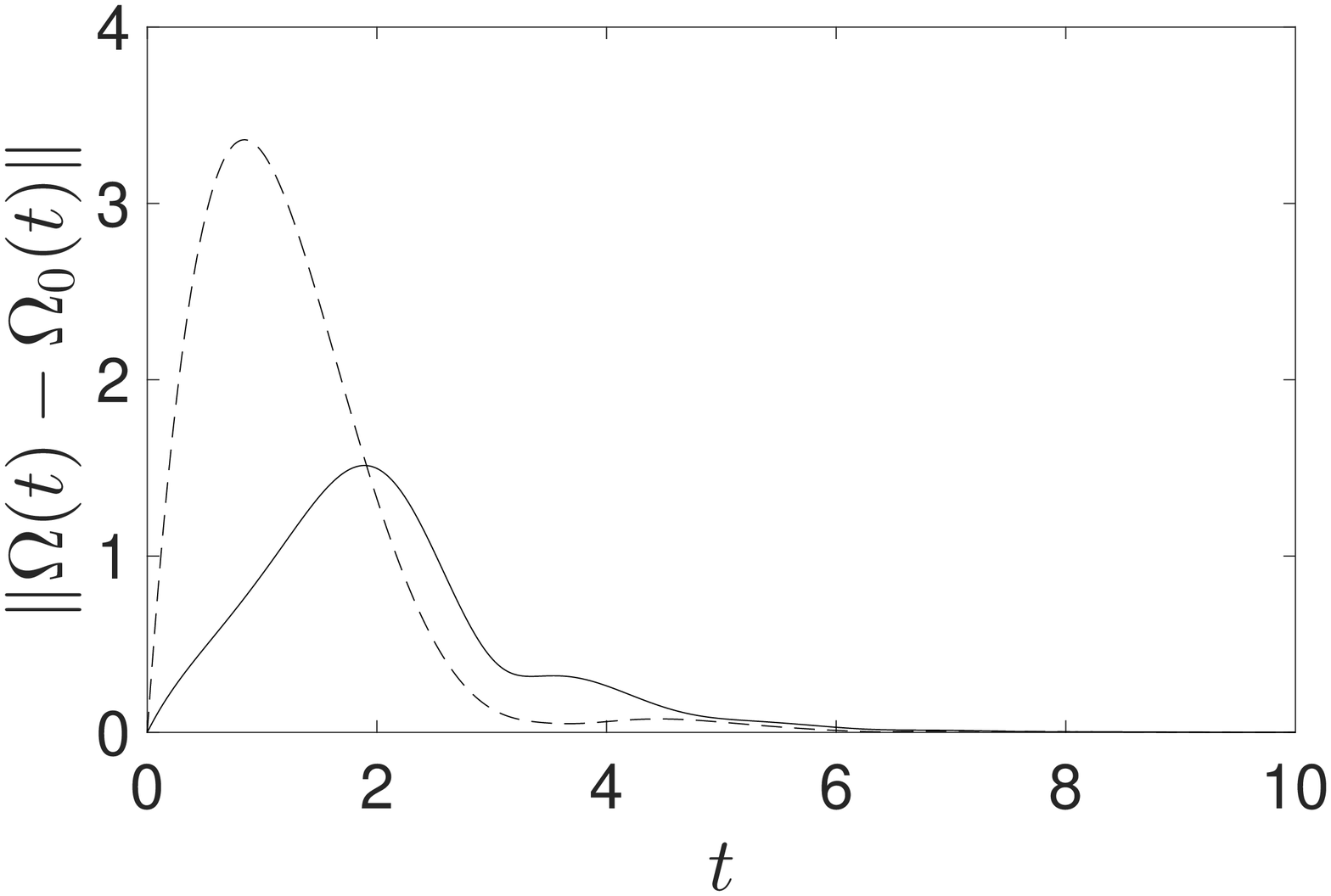}~~~\includegraphics[scale = 0.25]{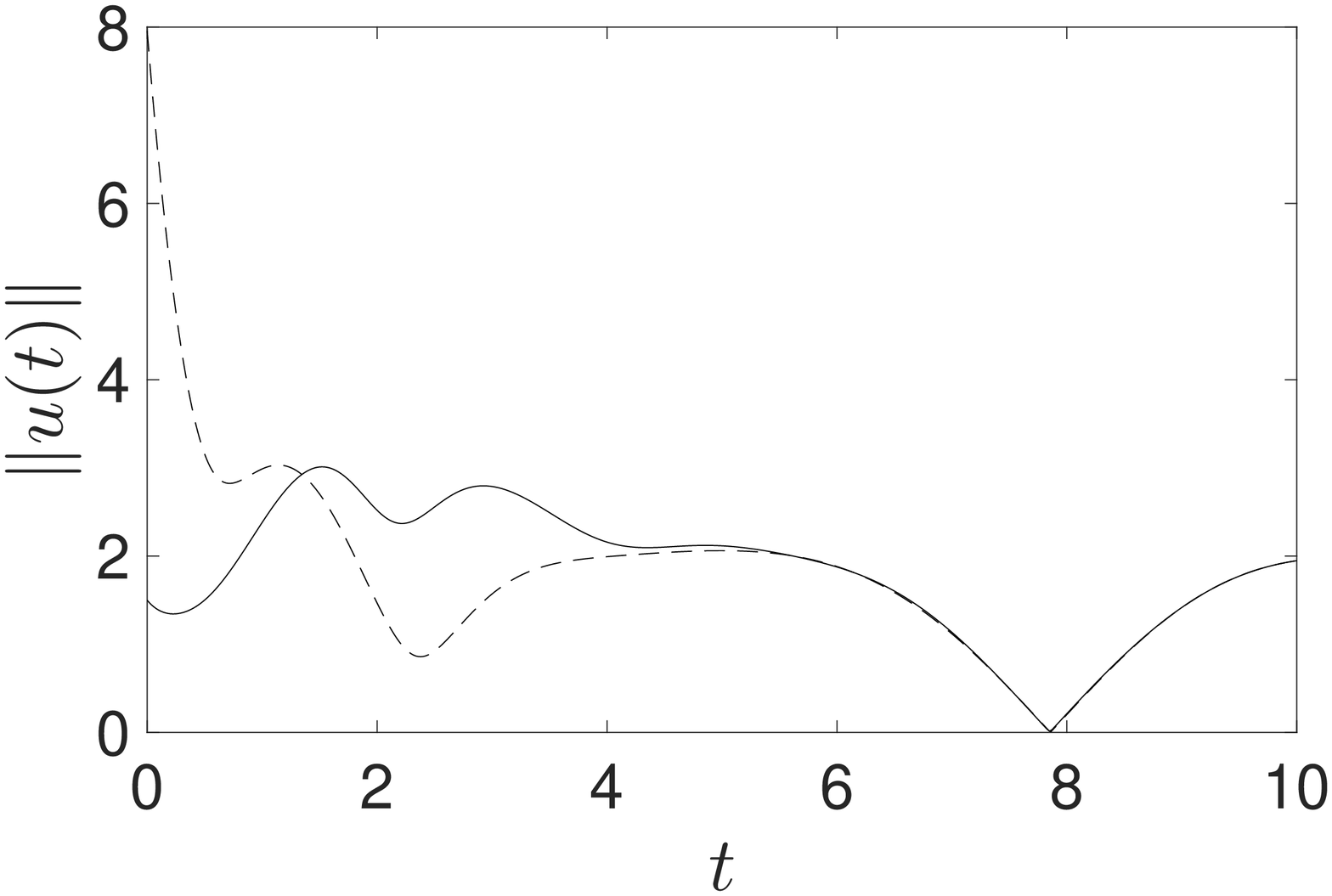}
\end{center}
\caption{\label{figure.tracking:rigid:Lee}  The simulation results for  tracking the reference $(R_0(t), \Omega_0(t))$ by the linear controller  \eqref{THE:tracking:control:rigid} with \eqref{delta:u:kpd:Omega0}  (solid)  and the nonlinear controller by Lee (dashed) for the rigid body system.}
\end{figure}

\subsection{Tracking Controller Design for  the Quadcopter System}
The equations of motion of the quadcopter system are given by
\begin{subequations}\label{quad:dynamics}
\begin{align}
\dot{R} &= {R} \hat {\Omega}, \label{equation:R}\\
{\mathbb I} \dot{\Omega} &= {\mathbb I} \Omega \times\Omega + \bf \tau,  \label{equation:ang:vel}\\
 \ddot{ x} &= -g { e}_3 + f {R} {e}_3, \label{equation:translation}
\end{align}
\end{subequations}
where  $x $ is the $\mathbb R^3$-vector for the position of  the quadcopter, $R$ is the $3\times 3$ rotation matrix for orientation, and $\Omega \in \mathbb R^3$ is the $\mathbb  R^3$-vector for body angular velocity. Here, $f \geq 0$ is the upward control thrust \textit{per mass} and $\tau = (\tau_1, \tau_2, \tau_3) \in \mathbb R^3$ is the control torque on the quadcopter expressed  in the body frame. The parameter  $g$ denotes the gravitational acceleration;  $\mathbb I$ is the $3\times 3$ moment of inertia matrix; and $e_3 =  (0,0,1)$. Although $f$ is a thrust per mass unit-wise, it shall be simply called a thrust in this paper. Refer to the book by Lee  et al.\cite{LeLeMc17} for the derivation of \eqref{quad:dynamics}.

Assume that the full state is available and apply the feedback
\begin{equation}\label{tau:from:u}
\tau =  - {\mathbb I}\Omega \times \Omega + {\mathbb I} u
\end{equation}
to transform the subsystem \eqref{equation:ang:vel} to 
\[
\dot \Omega = u,
\]
where $u  \in \mathbb R^3$ is the new control sub-vector replacing $\tau \in \mathbb R^3$. Extend dynamically the subsystem \eqref{equation:translation} by introducing a double integrator through the thrust  variable as follows:
\begin{equation}\label{f:extension}
\ddot f = q,
\end{equation}
where $q \in \mathbb R $ is now a new control variable, and $f$ and $\dot f$ are now regarded as state variables.
As done for the rigid body system, we embed $\SO$ to $\mathbb R^{3\times 3}$ and subtract $\nabla \tilde V$, with $\tilde V$ given in \eqref{def:V:tilde}, from the equations of motion   of the quadcopter  to get the following  equations of motion  in the ambient Euclidean space:
\begin{subequations} \label{modified:quad:extension}
\begin{align}
\dot{R} &= {R} \hat {\Omega} - k_eR (R^TR-I), \label{modified:quad:extension:a}\\
 \dot{\Omega} &= u,  \label{modified:quad:extension:b}\\
 \ddot{ x} &= -g { e}_3 + f {R} {e}_3, \label{modified:quad:extension:c}\\
\ddot f &= q. \label{modified:quad:extension:d}
\end{align}
\end{subequations}
Choose a reference trajectory 
\[
(R_0(t),\Omega_0(t), x_0(t), \dot x_0(t), f_0(t), \dot f_0(t))
\]
with $R_0(t) \in \SO$ for all $t\geq 0$, 
and a reference control signal 
\[
(u_0(t), q_0(t))
\]
such that they satisfy the equations of motion \eqref{modified:quad:extension}. It is understood that $\dot x_0(t)$ and $\dot f_0(t)$  are  the time derivatives of $x_0(t)$ and $f_0(t)$, respectively. It is further assumed that $\Omega_0(t), \dot \Omega_0(t), f_0(t), \dot f_0(t)$ and  $\ddot f_0(t)$ are bounded for $t\geq 0$, and  there is a constant $\delta >0$ such that 
\[
f_0(t) \geq \delta \quad \forall t\geq 0.
\]
Define the tracking error variables:
\begin{align*}
&\Delta R = R-R_0(t), \,\, \Delta \Omega = \Omega - \Omega_0(t), \,\, \Delta x = x - x_0(t),\\
&\Delta f = f-f_0(t), \,\, \Delta u = u - u_0(t), \,\, \Delta q = q - q_0(t).
\end{align*}
Then, linearize the system \eqref{modified:quad:extension} along the reference trajectory and use the state transformation given in \eqref{ZR0R} -- 
\eqref{def:Zk} replacing $\Delta R$, to obtain the following linearized system:
\begin{subequations}\label{linearized:quad}
\begin{align}
\dot Z_s &= [Z_s, \hat \Omega_0] - 2k_e Z_s,\label{linearized:quad:a}\\
\dot Z_k^\vee &= Z_k^\vee \times \Omega_0 + \Delta \Omega,  \label{linearized:quad:b}\\
\Delta \dot \Omega &= \Delta u, \label{linearized:quad:c}\\
\Delta \ddot x & = \Delta f R_0 e + f_0R_0(Z_s+Z_k) e_3,\label{linearized:quad:d}\\
\Delta \ddot f &= \Delta  q.\label{linearized:quad:e}
\end{align}
\end{subequations}
Retaining all the other state variables, we replace the state variable  $\Delta \Omega \in \mathbb R^3$, via \eqref{linearized:quad:b}, with  $\dot Z_k^\vee \in \mathbb R^3$ or  $\dot Z_k \in \mathfrak {so} (3)$.  Apply the feedback 
\begin{equation}\label{deltau:u:tile}
\Delta u = -(Z_k^\vee \times \Omega_0 + \Delta \Omega) \times \Omega_0  - Z_k^\vee \times \dot \Omega_0 + \tilde u, 
\end{equation}
so as to replace \eqref{linearized:quad:b} and \eqref{linearized:quad:c} with the following second-order equation:
\[
\ddot Z_k^\vee = \tilde u,
\]
where $\tilde u = (\tilde u_1, \tilde u_2, \tilde u_3)\in \mathbb R^3$ is the new control sub-vector replacing $\Delta u$.  Then, the  system \eqref{linearized:quad} is transformed to the following:
\begin{subequations}\label{linearized:quad:second}
\begin{align}
\dot Z_s &= [Z_s, \hat \Omega_0] - 2k_e Z_s, \label{linearized:quad:second:a}\\
\ddot Z_k^\vee &= \tilde u,\label{linearized:quad:second:b} \\
\Delta \ddot x & = \Delta f R_0 e + A_0(Z_s + Z_k) e_3,\label{linearized:quad:second:c}\\
\Delta \ddot f &= \Delta  q,\label{linearized:quad:second:d}
\end{align}
\end{subequations}
where the matrix-valued signal
\[
A_0(t) = f_0(t)R_0(t) \in \mathbb R^{3\times 3}
\]
is introduced for convenience.  Let
\begin{equation}\label{def:small:zk}
z_k =(z_{k1}, z_{k2}, z_{k3}) := Z_k^\vee \in \mathbb R^3
\end{equation}
so that 
\begin{equation}\label{def:small:zk:2}
Z_k = \begin{bmatrix}
0 & -z_{k3} & z_{k2} \\
z_{k3} & 0 & -z_{k1} \\
-z_{k2} & z_{k1} & 0
\end{bmatrix}.
\end{equation}

\begin{lemma}\label{lemma:coordinate:change}
The coordinate system
\begin{equation}\label{coord:sys:1}
(Z_s, Z_k^\vee,  \dot Z_k^\vee, \Delta x, \Delta \dot x, \Delta f, \Delta \dot f)
\end{equation}
can be globally replaced with
\begin{equation}\label{coord:sys:2}
(Z_s, \Delta x, \Delta \dot x, \Delta \ddot x, \Delta \dddot x,  z_{k3}, \dot z_{k3}).
\end{equation}
The coordinates $\Delta \ddot x$ and $\Delta \dddot x$ in \eqref{coord:sys:2} are expressed in terms of the coordinates \eqref{coord:sys:1} as 
\begin{align}
\Delta \ddot x &= (\Delta f R_0 + A_0Z_k + A_0Z_s) e_3,  \label{delta:ddotx}\\
\Delta \dddot x &= ( \Delta \dot f R_0 + A_0 \dot Z_k + \Delta f \dot R_0 + \dot A_0(Z_s+Z_k) + A_0 ([Z_s, \hat\Omega_0] - 2k_eZ_s) \big )e_3. \label{delta:dddotx} 
\end{align}
The coordinates $\Delta f$, $\Delta \dot f$, $z_{ki}$, $\dot z_{ki}$, $i=1,2$,  in \eqref{coord:sys:1} are expressed in terms of the coordinates \eqref{coord:sys:2} as 
\begin{align}
\begin{bmatrix}
z_{k2}, & z_{k1}, & \Delta f
\end{bmatrix}^T &= B_0^{-1}R_0^T  ( \Delta \ddot x - A_0Z_s e_3), \label{z:delta}\\
\begin{bmatrix}
\dot z_{k2},&  \dot z_{k1},& \Delta \dot f
\end{bmatrix}^T &= B_0^{-1}R_0^T(  \Delta \dddot x - ( \Delta f \dot R_0 + \dot A_0(Z_s+Z_k) + A_0 ([Z_s, \hat\Omega_0] - 2k_eZ_s))e_3 ),   \label{z:delta:dot}
\end{align}
where
\begin{equation}\label{def:B0}
B_0(t) = \operatorname{diag} [  f_0(t), -f_0(t), 1]\in \mathbb R^{3\times 3}.
\end{equation}
\begin{proof}
Differentiate \eqref{linearized:quad:second:c} with respect to $t$ and use  \eqref{linearized:quad:second:a} to replace $\dot Z_s$ with $[Z_s, \hat \Omega_0] - 2k_e Z_s$, so as to obtain the expression for $\Delta \dddot x$ in \eqref {delta:dddotx}. From the definition of the vector $z_k$ in \eqref{def:small:zk} or \eqref{def:small:zk:2},  $Z_ke_3 = z_{k2}e_1 - z_{k1}e_2$, where $e_1 = (1,0,0)$ and $e_2 = (0,1,0)$. Hence, it is straightforward to get \eqref{z:delta} and \eqref{z:delta:dot} from \eqref{delta:ddotx} and \eqref{delta:dddotx}, respectively.
\end{proof}
\end{lemma}
We express the  system \eqref{linearized:quad:second}  in the new coordinates \eqref{coord:sys:2} and  transform it via feedback to simple integrators as in the following theorem.
\begin{theorem}
The  system \eqref{linearized:quad:second}  is transformed to 
\begin{subequations}\label{final:linear:quad}
\begin{align}
\dot Z_s &= [Z_s, \hat \Omega_0] - 2k_e Z_s, \label{final:linear:quad:a}\\
\Delta x^{(4)} &= v, \label{final:linear:quad:b}\\
\ddot z_{k3} &= w, \label{final:linear:quad:c}
\end{align}
\end{subequations}
where $(v,w) \in \mathbb R^3 \times \mathbb R$ is the new control vector, 
by the feedback
\begin{subequations}\label{tildeu:from:final}
\begin{align}
\tilde u_3 &= w,\label{tildeu:from:final:a}\\
\begin{bmatrix}
\tilde u_2,& \tilde u_1,& \Delta q
\end{bmatrix}^T&= B_0^{-1}R_0^T ( v - Ce_3) \label{tildeu:from:final:b}
\end{align}
\end{subequations}
where
\begin{align}
C &= 2\Delta \dot f \dot R_0 + 2\dot A_0(\dot Z_s +\dot Z_k) + \Delta f \ddot R_0 + \ddot A_0 (Z_s+Z_k) +A_0([\dot Z_s, \hat \Omega_0] + [Z_s, \hat{\dot \Omega}_0] - 2k_e\dot Z_s). \label{def:C:matrix}
\end{align}
In the above expression of $C$,  $\dot Z_s$ is understood as $ [Z_s, \hat \Omega_0] - 2k_e Z_s$.
\begin{proof}
Differentiate \eqref{delta:dddotx} with respect to $t$ and simplify the result using the equations of motion in \eqref{linearized:quad:second}  to obtain
\begin{align*}
\Delta x^{(4)} &= \Delta \ddot f  R_0 e_3+ A_0\ddot Z_ke_3  + Ce_3\\
&= R_0B_0 (\tilde u_2 e_1 + \tilde u_1 e_2 + \Delta q e_3)  + Ce_3,
\end{align*}
with $B_0$ and $C$ defined in \eqref{def:B0} and \eqref{def:C:matrix}, respectively. It  is transformed to \eqref{final:linear:quad:b} by the feedback \eqref{tildeu:from:final:b}. Equation \eqref{final:linear:quad:c} is obtained by taking the inner product of \eqref{linearized:quad:second:b} with $e_3$ and using \eqref{tildeu:from:final:a}.
\end{proof}
\end{theorem}

\begin{proposition}\label{proposition:quad:linear:control}
Take any four matrices $K_0, K_1, K_2, K_3 \in \mathbb R^{3\times 3}$ such that the polynomial
\[
\det ( \lambda^4 I+ \lambda^3K_3 + \lambda^2 K_2 + \lambda K_1 + K_0 )
\]
is a Hurwitz polynomial in $\lambda$, and take any two positive numbers $a_1$ and $a_0$. Then, the feedback controller
\begin{align}
 v &= -K_3 \Delta \dddot x - K_2\Delta \ddot x - K_1\Delta \dot x - K_0\Delta x, \label{quad:v}\\
 w  &= -a_1 \dot z_{k3} - a_0 z_{k3}\label{quad:w}
\end{align}
makes the origin exponentially stable for the system  \eqref{final:linear:quad}. 
\begin{proof}
The exponential stability of the $Z_s$ dynamics \eqref{final:linear:quad:a} has been already shown in the proof of Proposition \ref{proposition:delta:1}. It is trivial to show the exponential stability of the origin for the subsystem \eqref{final:linear:quad:b} and \eqref{final:linear:quad:c} with the proposed controller.
\end{proof}
\end{proposition}
Notice that  the controller  in  \eqref{quad:v} and \eqref{quad:w} can be expressed in terms of the original variables via Lemma \ref{lemma:coordinate:change} and equations \eqref{ZR0R}, \eqref{def:Zs} and \eqref{linearized:quad:b}.

\begin{proposition}\label{proposition:quad:linear:control:b}
Take any five  matrices $K_0, K_1, K_2, K_3, K_I \in \mathbb R^{3\times 3}$ such that the polynomial
\[
\det (\lambda^5 I+ \lambda^4K_3 + \lambda^3 K_2 + \lambda^2 K_1 + \lambda K_0  + K_I )
\]
is a Hurwitz polynomial in $\lambda$, and take any three  numbers $a_1, a_0, a_I$ such that the polynomial
\[
\lambda^3 + a_1\lambda^2 + a_0\lambda + a_I 
\]
is Hurwitz. Then, the feedback controller
\begin{align*}
 v &=\! -K_3 \Delta \dddot x - K_2\Delta \ddot x - K_1\Delta \dot x - K_0\Delta x - K_I \! \int_0^t \! \Delta x(\tau) d\tau , \\
 w  &= -a_1 \dot z_{k3} - a_0 z_{k3} - a_I\int_0^t z_{k3}(\tau) d\tau
\end{align*}
makes the origin exponentially stable for the system  \eqref{final:linear:quad}. 
\begin{proof}
Trivial.
\end{proof}
\end{proposition}

After  a controller $(v,w)$ is designed as in Propositions  \ref{proposition:quad:linear:control} and \ref{proposition:quad:linear:control:b}, the controller $(\tilde u, \Delta q)$ in \eqref{tildeu:from:final} is computed. Then,  $\Delta u$ is computed via \eqref{deltau:u:tile}, which produces the  control torque $\tau$ in \eqref{tau:from:u} with $u = u_0(t) + \Delta u$ and the control thrust $f$ via \eqref{f:extension} with   $q = q_0(t) + \Delta q$. 

\begin{theorem}\label{theorem:quad:tarcking:mfd}
The controller $(\tau, f)$  designed as above enables the  quadcopter system \eqref{quad:dynamics} to exponentially track the reference trajectory $(R_0(t), \Omega_0(t), x_0(t), \dot x_0(t))$.
\begin{proof}
It is easy to prove that the origin is exponentially stable for the linear system \eqref{linearized:quad} with the controller $(\Delta u, \Delta q)$ designed as described above. By Theorem \ref{theorem:Khalil}, the controller $ (u, q)$ designed as described above enables the extended quadcopter system  \eqref{modified:quad:extension} to exponentially track the reference trajectory $(R_0(t), \Omega_0(t), x_0(t), \dot x_0(t), f_0(t), \dot f_0(t))$ from which the present theorem follows.
\end{proof}
\end{theorem}

\begin{remark}
The controllers proposed in the paper by Goodarzi  et al.\cite{GoLeLe15} have two separate modes: attitude controlled flight mode and position controlled flight mode. In contrast, our controllers have the merit to simultaneously control both the attitude and the position of  quadcopter. 
\end{remark}

\begin{remark}
Our controllers have no singularity  since we use only one single global Cartesian coordinate system, whereas the controller proposed by Mellinger and  Kumar\cite{MeKu11} would become singular when the roll angle becomes $\pm \pi/2$, which purely comes from the use of an Euler angle coordinate system. This shows the merit of our method that utilizes one single global Cartesian  coordinate system in the  ambient Euclidean space. It will be interesting to re-do the work by Mellinger and  Kumar\cite{MeKu11} in this framework.
\end{remark}

\begin{remark}
Although the dynamic extension \eqref{f:extension} is simple, it has the drawback that the non-negative sign of $f(t)$ may not be preserved along the trajectory even with a positive initial value $f(0)>0$. To remedy this, the following  dynamic extension
\begin{equation}\label{new:ext}
\dot f = fh, \quad \dot h = q
\end{equation}
was proposed in the paper by Chang and Eun\cite{ChEu17} to replace \eqref{f:extension}, where $h$ is an added state variable replacing $\dot f$.  It is easy to verify that this extension preserves the positive sign of $f(t)$ when $f(0)>0$. The linearization of \eqref{new:ext} along the reference trajectory is computed as
\[
\Delta \dot f = \Delta f h_0 + f_0 \Delta h, \quad \Delta \dot h = \Delta q,
\]
and it shall  replace \eqref{linearized:quad:e} in the linearization of the quadcopter dynamics, where $h_0(t) = \dot f_0(t)/f_0(t)$ and $\Delta h = h - h_0(t)$. It is left to the reader to verify that  with the extension \eqref{new:ext} the consequent linearized quadcopter system can also be transformed to   \eqref{final:linear:quad} via an appropriate feedback control law.
\end{remark}

We now run a simulation to demonstrate a good tracking performance of the proposed controller 
$u = u_0(t) + \Delta u$ and    $q = q_0(t) + \Delta q$ with $\Delta u$ , $\tilde u$,  $\Delta q$, $v$ and $w$ given in  \eqref{deltau:u:tile}, \eqref{tildeu:from:final},  \eqref{quad:v} and \eqref{quad:w},  for the extended quadcopter system \eqref{modified:quad:extension} with $k_e=1$. Choose a reference trajectory for \eqref{modified:quad:extension} as follows:  $R_0(t)$, $\Omega_0(t)$ and $u_0(t)$  are given in \eqref{tracking:rigid:R0} --  \eqref{tracking:rigid:u0}, and $x_0(t)$ and $f_0(t)$ are given as
\begin{align*}
&x_0(t)  = g\begin{bmatrix} \frac{1}{2}t^2 + \frac{4}{9}\sin t - \frac{1}{2}\sin^2 t + \frac{2}{9}\sin t \cos^2 t \\
 \frac{4}{9} - \frac{4}{9}\cos t - \frac{1}{2} \cos t \sin t - \frac{2}{9} \cos t \sin^2 t \\
\frac{1}{2}\sin^2t \end{bmatrix},\\
 &f_0(t) = 2g.
\end{align*}
Choose the following initial condition for \eqref{modified:quad:extension}:
 \begin{align*}
R(0)&= \exp(0.25\pi\hat e_2),\quad
\Omega(0) = (0,0,0),\\
x(0) &= (-0.5g, -0.5g, 0), \quad
\dot x(0) = (0, 0,0),\\
f(0) &= 2g, \quad \dot f(0) = 0,
 \end{align*}
where $R(0)$ is a rotation through $\pi/4$ radians about the axis $e_2 = (0,1,0)$.
By scaling $x$ by $g$, we may assume that $
g = 1$.
Choose the following values of control parameters:
\[
K_3 = 8I, \quad  K_2 = 32I, \quad K_1 = K_0 = 64I,
\quad a_1 = 8, \quad a_0 = 20
\]
for \eqref{quad:v} and \eqref{quad:w}, so that the poles of the tracking error dynamics  \eqref{final:linear:quad:b} for $\Delta x$ are all located  at $-2\pm j2$ and the poles of  \eqref{final:linear:quad:c} for $z_{k3}$ are located at $-4\pm j2$. Apply the resulting  controller $(u,q)$ to \eqref{modified:quad:extension}. The tracking errors and the control thrust $f$ are plotted in Figure \ref{figure.tracking:quad}. The tracking errors all converge to zero as $t\rightarrow \infty$, and the control thrust $f$ converges to the reference thrust $f_0(t) = 2$ as $t\rightarrow \infty$.  To test robustness of the controller to disturbance, we now add disturbance terms to \eqref{modified:quad:extension:b} and \eqref{modified:quad:extension:c} as follows:
\begin{align*}
\dot \Omega &= u + R^Td,\\
 \ddot{ x} &= -g { e}_3 + f {R} {e}_3 +  d,
\end{align*}
where $d(t) = \sin(2\pi(t-3))(1,1,1)$ if $ 3 \leq t \leq 4$,  and $d(t) = 0$ otherwise.   We run a simulation with the same controller without any compensation for the disturbance. The simulation result is plotted in Figure \ref{figure.tracking:quad:dis}, where the two dotted vertical lines denote the start time and end time of the disturbance. We can see that the tracking degrades from $t=3$ till approximately $t=4.2$ due to the effect of disturbance and then gets back to the exponentially convergent mode.  This result shows robustness of our tracking controller to disturbance.

\begin{figure}[h]
\begin{center}
\includegraphics[scale = 0.4]{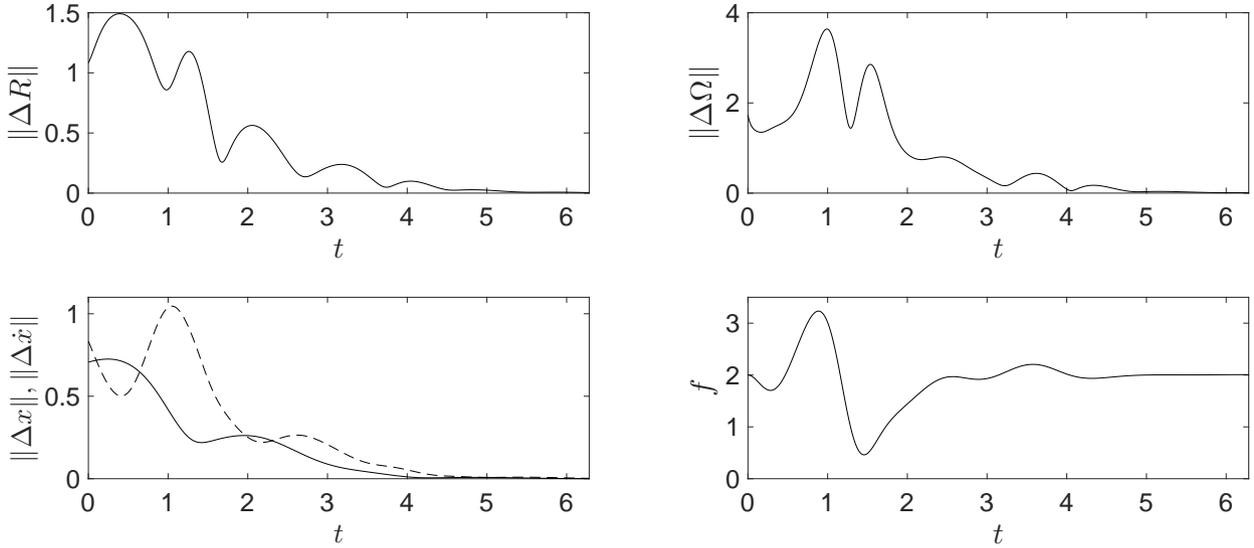}
\end{center}
\caption{\label{figure.tracking:quad}  The trajectory of the  tracking errors and the thrust variable of the quadcopter system \eqref{modified:quad:extension} with the linear controller  described in Theorem \ref{theorem:quad:tarcking:mfd}. In the left bottom plot, the solid line is the trajectory of $\| \Delta x(t)\|$ and the dashed line that of $\| \Delta \dot x(t)\|$. }
\end{figure}

\begin{figure}[h]
\begin{center}
\includegraphics[scale = 0.4]{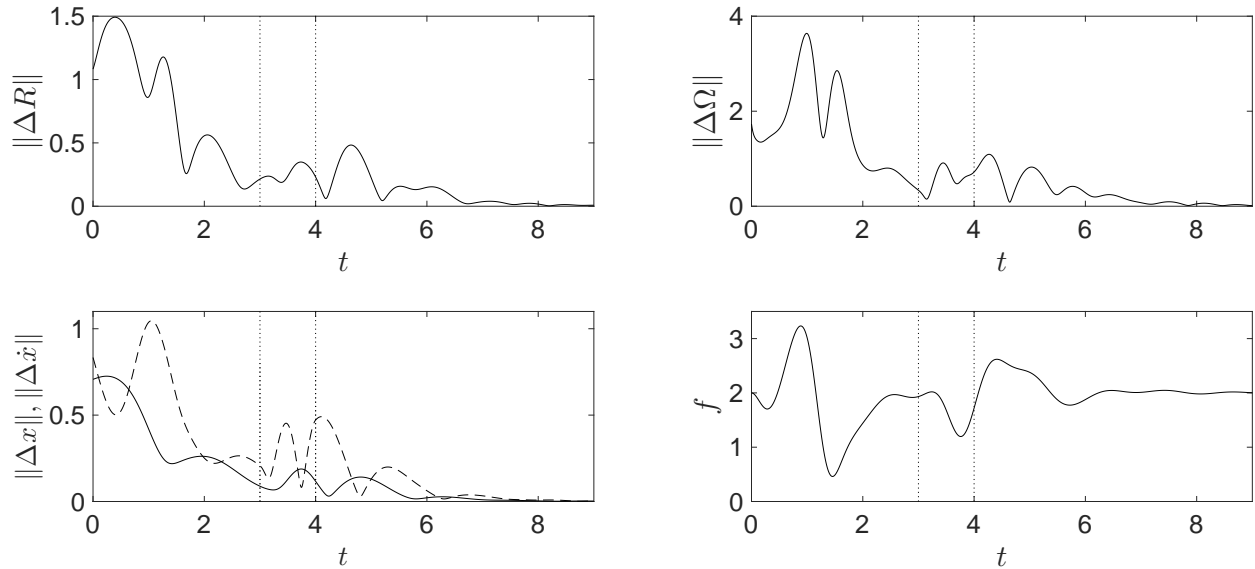}
\end{center}
\caption{\label{figure.tracking:quad:dis}  The trajectory of the  tracking errors and the thrust variable of the quadcopter system \eqref{modified:quad:extension} with the linear controller  described in Theorem \ref{theorem:quad:tarcking:mfd} in the presence of an unknown disturbance during the time interval, $3 \leq t \leq 4$. The two dotted vertical lines denote the time interval $[3,4]$. In the left bottom plot, the solid line is the trajectory of $\| \Delta x(t)\|$ and the dashed line that of $\| \Delta \dot x(t)\|$. }
\end{figure}

\section{Conclusion}
We have presented a method to design controllers in Euclidean space for systems defined on manifolds. The idea is to embed the state-space manifold $M$ of a given control system  to some Euclidean space  $\mathbb R^n$,  extend the system from $M$ to the ambient space $\mathbb R^n$, and modify it outside $M$ to add transversal stability to $M$ in the final dynamics in $\mathbb R^n$. We then design controllers for the final system in the ambient Euclidean space $\mathbb R^n$ and restrict the controllers to $M$ after the synthesis.  Since the controller synthesis is carried out in Euclidean space in this framework, it has the  merit that only one single global Cartesian coordinate system in the ambient Euclidean space is used and all possible controller design methods on $\mathbb R^n$, including the linearization method, can be rigorously applied for controller synthesis.   This method is successfully applied to the  tracking problem for the following two benchmark systems: the fully actuated rigid body system and the quadcopter drone system.  As future work, we plan to consider control constraints such as saturation in the proposed method for which  the technique  developed by Su et al. \cite{SuFaChLu17} is expected to be effective. We also plan to study robustness of the proposed method with respect to measurement errors.


\section*{Acknowledgment}
\begin{ack}                               % Place acknowledgements
This research has been in part supported by KAIST  under grant G04170001 and  by the ICT R\&D program of MSIP/IITP [2016-0-00563, Research on Adaptive Machine Learning Technology Development for Intelligent Autonomous Digital Companion]. 
\end{ack}

                                        % in the appendices.
\end{document}